\documentclass[11pt,a4paper,leqno]{article}
\usepackage{a4wide}
\usepackage{latexsym}
\usepackage{amsmath}
\usepackage{amssymb}
\usepackage{stackrel} 
\usepackage{color}
\usepackage{graphicx}
\usepackage{multicol}

\usepackage{lipsum}

\newcommand\blfootnote[1]{%
  \begingroup
  \renewcommand\thefootnote{}\footnote{#1}%
  \addtocounter{footnote}{-1}%
  \endgroup
}

\usepackage{amsfonts, amsthm, amscd}
\usepackage{mathrsfs}

\newcommand{\BC}{\mathbb{C}}
\newcommand{\BR}{\mathbb{R}}
\newcommand{\BN}{\mathbb{N}}

\newcommand{\bnu}{\boldsymbol{\nu}}
\newcommand{\bz}{\boldsymbol{z}}

\newtheorem{defin}{Definition}
\newtheorem{lemma}{Lemma}
\newtheorem{prop}{Proposition}
\newtheorem{theo}{Theorem}
\newtheorem{corol}{Corollary}
\begin{document}
\title{Maillet type theorem for nonlinear totally characteristic partial differential equations}

\author{{\bf Alberto Lastra$^1$, Hidetoshi Tahara$^2$}\\
\phantom{mmm}\\
$^1$University of Alcal\'{a}, Departamento de F\'{i}sica y Matem\'{a}ticas,\\
Ap. de Correos 20, E-28871 Alcal\'{a} de Henares (Madrid), Spain,\\
{\tt alberto.lastra@uah.es}\\
\phantom{mmm}\\
 $^2$Sophia University, Department of Information and Communication 
       Sciences,\\
Kioicho, Chiyoda-ku, Tokyo 102-8554, Japan.\\
{\tt h-tahara@sophia.ac.jp }}

\date{}
\maketitle

\blfootnote{A. Lastra is partially supported by the project MTM2016-77642-C2-1-P of Ministerio de Econom\'ia y Competitividad, Spain. H. Tahara is partially supported by JSPS KAKENHI Grant
Number 15K04966.}

\begin{abstract}
\noindent The paper discusses a holomorphic nonlinear singular partial 
differential equation
$(t \partial_t)^mu=F(t,x,\{(t \partial_t)^j
     \partial_x^{\alpha}u \}_{j+\alpha \leq m, j<m})$ under the 
assumption that the equation is of nonlinear totally 
characteristic type. By using the Newton Polygon at $x=0$,
the notion of the irregularity at $x=0$ of the equation is defined.
In the case where the irregularity is greater than one, it is proved 
that every formal power series solution belongs to a suitable formal 
Gevrey class. The precise bound of the order of the formal Gevrey 
class is given, and the optimality of this bound is also proved in a 
generic case.
\end{abstract}

\section{Introduction}

In 1903, Maillet \cite{maillet} showed that all formal power
series solutions of nonlinear algebraic ordinary differential 
equations are in some formal Gevrey class (see Definition~\ref{defi12}). In this 
paper, we achieve a Maillet type theorem for general nonlinear 
totally characteristic type partial differential equations.
\par

We first fix some notations, used through the present work. 

\noindent We write $\BN=\{0,1,2,\ldots \}$ and $\BN^*=\{1,2,\ldots \}$. For $m \in \BN^*$, we consider the sets
$I_m=\{(j,\alpha) \in \BN \times \BN \,;\, j+\alpha \leq m, j<m \}$, and $I_m(+)= \{(j,\alpha) \in I_m \,;\, \alpha>0 \}$. The pair $(t,x)$ stands for the variables in $\BC_t \times \BC_x$, and $\bz=\{z_{j,\alpha}\}_{(j,\alpha) \in I_m}$ in $\BC^N$ (with $N=\# I_m=m(m+3)/2$).

\noindent $\BC[[x]]$ denotes the ring of formal power series in $x$, and $\BC[[t,x]]$ denotes the ring of formal power 
series in $(t,x)$. Similarly, $\BC \{x\}$ denotes the ring of 
convergent power series in $x$, and $\BC \{t,x\}$ denotes the ring 
of convergent power series in $(t,x)$. 

\noindent Given $f(x)=\sum_{l \geq 0} f_lx^l \in \BC[[x]]$, we write
$f(x) \gg 0$ if $f_l \geq 0$ for all $l \geq 0$ and $|f|(x)$ denotes the formal power series $\sum_{j \geq 0}|f_j|x^j$.

\noindent For $R>0$ we write $D_R=\{x \in \BC \,;\, |x|<R \}$,
and $\overline{D}_R=\{x \in \BC \,;\, |x| \leq R \}$. We denote by
${\mathcal O}(D_R)$ the set of all holomorphic functions on $D_R$, 
and by ${\mathcal O}(\overline{D}_R)$ the set of all holomorphic 
functions in a neighborhood of $\overline{D}_R$. 

\noindent Given $x\in\mathbb{R}$, we denote $[x]$ the integer part of $x$, and $[x]_+= \max\{x,0 \}$.

Let $F(t,x,\bz)$ be a 
function defined in a polydisk $\Delta$ centered at the origin 
of $\BC_t \times \BC_x \times \BC_{\bz}^N$. In this paper, we consider the following nonlinear  partial differential equation
\begin{equation}\label{e11}
    (t \partial_t)^mu
   =F \bigl(t,x, \{(t \partial_t)^j
        \partial_x^{\alpha}u \}_{(j,\alpha) \in I_m} \bigr)
\end{equation}
under the assumptions
\begin{itemize}
\item[$\left.\hbox{A}_1\right)$] $F(t,x,\bz)$ is holomorphic in $\Delta$,
\item[$\left.\hbox{A}_2\right)$] $F(0,x,\boldsymbol{0}) \equiv 0$ in $\Delta_0=\Delta \cap \{t=0, \bz=\boldsymbol{0} \}$.
\end{itemize}
 
Under the previous assumptions, $F(t,x,\bz)$ can be expressed in the form 
$$
   F(t,x,\bz)=a(x)t + \sum_{(j,\alpha) \in I_m}
             b_{j,\alpha}(x) z_{j,\alpha} + R_2(t,x,\bz)
$$
where $R_2(t,x,\bz)$ is a holomorphic function on $\Delta$ 
whose Taylor expansion in $(t,\bz)$ has the form
\begin{equation}\label{e12}
     R_2(t,x,\bz)= \sum_{i+|\bnu| \geq 2}
         a_{i,\bnu}(x) t^i \bz^{\bnu},
\end{equation}
where $\bnu=\{\nu_{j,\alpha} \}_{(j,\alpha) \in I_m} \in \BN^N$,
$|\bnu|=\sum_{(j,\alpha) \in I_m}\nu_{j,\alpha}$ and 
$\bz^{\bnu}=\prod_{(j,\alpha) \in I_m}
         {z_{j,\alpha}}^{\nu_{j,\alpha}}$.
\par

Different studies have been developed in the study of equation (\ref{e11}), which can be structured into three different blocks:
\begin{itemize}
\item Type 1: $b_{j,\alpha}(x) \equiv 0$ on $\Delta_0$ for any $(j,\alpha) \in I_m(+)$, 
\item Type 2: $b_{j,\alpha}(0) \ne 0$ for some $(j,\alpha) \in I_m(+)$, 
\item Type 3: Cases not considered above.
\end{itemize}

Equation (\ref{e11}) under Type 1 condition deals with the so called nonlinear Fuchsian type 
partial differential equations. It has been studied by several authors such as
Baouendi-Goulaouic~\cite{baouendi}, G\'erard-Tahara~\cite{gt1,gt2}, 
Madi-Yoshino~\cite{yoshino}, Tahara-Yamazawa~\cite{yamazawa} and 
Tahara-Yamane~\cite{yamane}. A Gousart problem appears when considering equations within Type 2: G\'erard-Tahara~\cite{gt3} discussed a particular class of equations in Type 2 and proved the existence of holomorphic solutions and also singular solutions of (\ref{e11}). An equation of the form (\ref{e11}) under the conditions in Type 3 is called a nonlinear totally 
characteristic type partial differential equation. The main thema of this paper is to discuss Type 3 under the following condition:
\begin{itemize}
\item[$\left.\hbox{A}_3\right)$] $b_{j,\alpha}(x)=O(x^{\alpha})$ (as 
          $x \longrightarrow 0$)
     for any $(j,\alpha) \in I_m(+)$.
\end{itemize}
Under this condition, we write $b_{j,\alpha}(x):=x^{\alpha}c_{j,\alpha}(x)$ for some holomorphic 
functions $c_{j,\alpha}(x)$ in a neighborhood of $x=0 \in \BC$.
We set
\begin{align}
    &C(x; \lambda, \rho)= \lambda^m
       -\sum_{(j,\alpha) \in I_m} c_{j,\alpha}(x)
        \lambda^j \rho (\rho-1) \cdots (\rho-\alpha+1), \label{e13}\\
    &L(\lambda, \rho)= C(0; \lambda, \rho). \label{e14}
\end{align}
Then, equation (\ref{e11}) is written in the form
\begin{equation}\label{e15}
    C(x;t \partial_t,x \partial_x)u
    =a(x)t+ R_2\bigl(t,x, \{(t \partial_t)^j
        \partial_x^{\alpha}u \}_{(j,\alpha) \in I_m} \bigr).
\end{equation}
\begin{prop}\label{prop11} Assume that the non-resonance condition
$$
    L(k,l) \ne 0 \quad 
         \mbox{for any $(k,l) \in \BN^* \times \BN$}
    \leqno{\quad \,{\rm (N)}}
$$
is satisfied. Then, equation (\ref{e11}) admits a unique formal power series
solution $u(t,x) \in \BC[[t,x]]$, with $u(0,x) \equiv 0$.
\end{prop}

About the convergence of this formal solution, nice results can be found in Chen-Tahara~\cite{chen} and Tahara~\cite{poincare}. In the case where the formal solution is divergent, to measure
the rate of divergence we use the following formal Gevrey classes:
\begin{defin}\label{defi12}
  (i) Let $s \geq 1$, $\sigma \geq 1$.
We say that the formal series
$f(t,x)=\sum_{k \geq 0,l \geq 0}a_{k,l}t^k x^l \in \BC[[t,x]]$
belongs to the formal Gevrey class $G\{t,x \}_{(s,\sigma)}$ 
of order $(s,\sigma)$ if the power series
$$
    \sum_{k \geq 0, l \geq 0}
     \frac{a_{k,l}}{k!^{s-1} l!^{\sigma-1}} t^k x^l 
$$
is convergent in a neighborhood of $(0,0) \in \BC_t \times \BC_x$.
\par
   (ii) Similarly, we say that the formal series
$f(x)=\sum_{l \geq 0}a_{l}x^l \in \BC[[x]]$
belongs to the formal Gevrey class $G\{x \}_{\sigma}$ 
of order $\sigma$ if the power series
$$
    \sum_{l \geq 0}
     \frac{a_{l}}{l!^{\sigma-1}} x^l 
$$
is convergent in a neighborhood of $0 \in \BC_x$.
\end{defin}

Let $u(t,x)$ be the formal solution of (\ref{e11}), whose existence is guaranteed in Proposition~\ref{prop11}, under condition (N). The main aim in the present study is to answer the following natural questions:
\begin{itemize}
\item[a)] Does $u(t,x)$ belong to $G\{t,x \}_{(s,\sigma)}$ for some $(s,\sigma)$ ?
\item[b)] If the answer is affirmative, determine the precise bound of the order
$(s,\sigma)$.
\end{itemize}

    In the case $m=1$, this problem was solved by Chen-Luo-Tahara
\cite{CLT}; in the case $m \geq 2$, Chen-Luo~\cite{CL2} has given 
a partial answer. The purpose of this paper is to give a final result 
in the general case. 
\par
Similar equations and problems are studied also in Chen-Luo~\cite{CL}, Shirai~\cite{shirai1, shirai2} and Yamazawa~\cite{yama}.

The paper is organized as follows. In Section~\ref{sec2}, we describe the construction of the Newton polygon associated to the main equation, and related elements and properties. In Section~\ref{sec24}, we state the two main results of the present work, namely Theorem~\ref{maintheorem}, and Theorem~\ref{teo27}. In Section~\ref{sec3}, we present some preparatory discussions which are needed in the proof of (ii) of Theorem~\ref{maintheorem}. After that, in Section~\ref{sec4} we give a proof of (ii) of Theorem~\ref{maintheorem}, and in Section~\ref{secproofmaintheo} we give a proof of Theorem~\ref{teo27}. In the last section, Section~\ref{sec6}, we give a slight generalization of the above results.

\section{Main result}\label{sec2}

In this section, we first recall the definition of
the Newton polygon ${\mathcal N}_0$ of equation (\ref{e11}) at $x=0$ and the 
generalized Poincar\'e condition (GP), in~\cite{poincare}. Then 
we define a notion of the irregularity $\sigma_0$ of (\ref{e11}) 
at $x=0$. After that, we give our main theorem, and the 
optimality of our condition.

\subsection{On the Newton polygon associated to the main equation}
\par
\medskip
Assume the conditions $\mbox{A}_1)$, $\mbox{A}_2)$ and 
$\mbox{A}_3)$ hold, and define $c_{j,\alpha}(x)$ ($(j,\alpha) \in I_m$) as in~(\ref{e13}).  Set $c_{m,0}(x)=-1$, and 
$$    \Lambda_0 = \{(m,0) \} \cup \{(j,\alpha) \in I_m \,;\,
             c_{j,\alpha}(0) \ne 0 \}.$$
For $(a,b) \in \BR^2$, we write 
$C(a,b)=\{(x,y) \in \BR^2 \,;\, x \leq a, y \leq b \}$.  Then, the
{\it Newton polygon ${\mathcal N}_0$ at $x=0$} of equation (\ref{e11}) 
is defined by the convex hull of the union of sets $C(j, \alpha)$ 
($(j,\alpha) \in \Lambda_0$) in $\BR^2$; that is, 
$$
   {\mathcal N}_0 = \mbox{the convex hull of} 
        \enskip \bigcup_{(j,\alpha) \in \Lambda_0} 
         C(j, \alpha) 
$$
(see Section 2 in~\cite{poincare}). An example of Newton polygon is illustrated in Figure~\ref{Fig1}.

\begin{figure}
	\centering
		\includegraphics{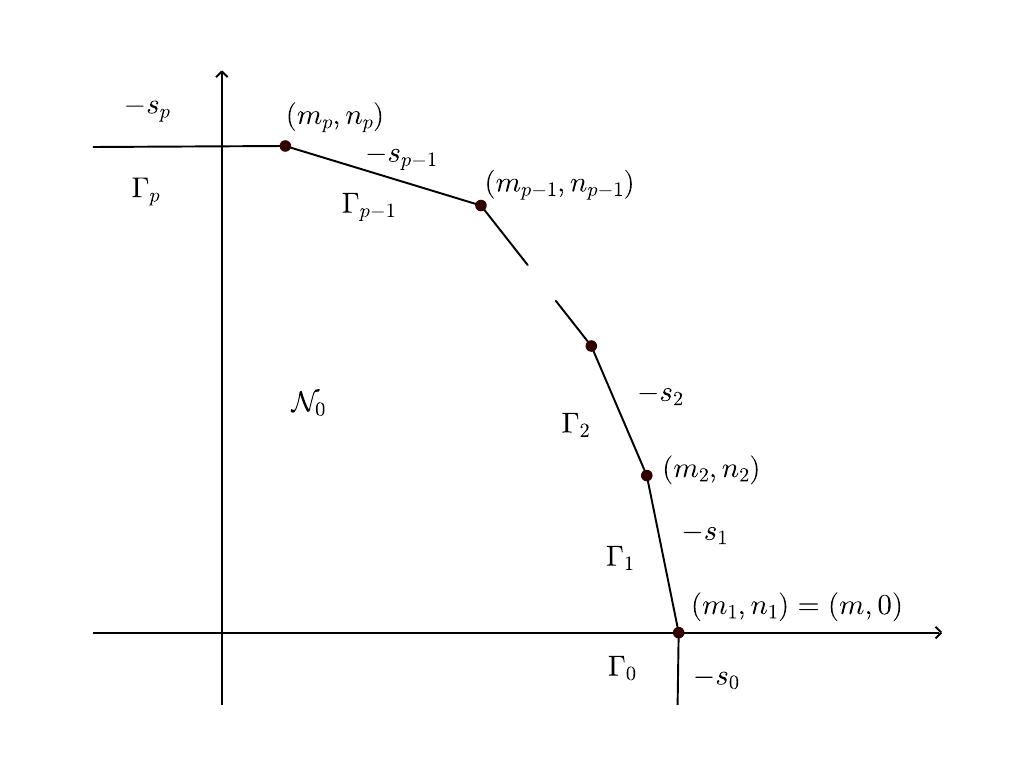}
\caption{Newton polygon ${\mathcal N}_0$ at $x=0$}\label{Fig1}
\end{figure}

   As is seen in Figure~\ref{Fig1}, the vertices of ${\mathcal N}_0$ 
consist of $p$ points
$$
    (m_1,n_1)=(m,0), \enskip (m_2,n_2), 
    \enskip \cdots, \enskip (m_{p-1},n_{p-1}),
    \enskip (m_p,n_p),
$$
and the boundary of ${\mathcal N}_0$ consists of a vertical half 
line $\Gamma_0$, $(p-1)$-segments 
$\Gamma_1, \Gamma_2, \ldots, \Gamma_{p-1}$, and a horizontal 
half line $\Gamma_p$.  We denote the slope of $\Gamma_i$ by 
$-s_i$ ($i=0,1,2,\ldots,p$), and have 
$$
    s_0=\infty >s_1>s_2> \cdots >s_{p-1}>s_p=0.
$$ 
\par
   Let us recall the following definition (see Definition 1 
in~\cite{poincare}).

\begin{defin}\label{defi21}
We say that equation (\ref{e11}) has a regular singularity at $x=0$ if the following condition is 
satisfied:
$$
       \mbox{if $c_{j,\alpha}(0)=0$ and 
          $c_{j,\alpha}(x) \not\equiv 0$, 
               then $(j,\alpha) \in {\mathcal N}_0$}.
        \leqno{\quad \,{\rm (R)}}
$$
Otherwise, that is, if (R) is not satisfied then we say that 
equation (\ref{e11}) has an irregular singularity at $x=0$.
\end{defin}

\subsection{Generalized Poincar\'e condition}\label{sec22}

For $1 \leq i \leq p-1$ we define the characteristic polynomial on 
$\Gamma_i$ by
$$    P_i(X) = \sum_{(j,\alpha) \in \Lambda_0 \cap \Gamma_i}
           c_{j,\alpha}(0) X^{j-m_{i+1}}= c_{m_i,n_i}(0)X^{m_i-m_{i+1}} + \cdots
            + c_{m_{i+1},n_{i+1}}(0).$$
We denote $\lambda_{i,h}$
($1 \leq h \leq m_i-m_{i+1}$) the roots of $P_i(X)=0$ which are 
called the characteristic roots on $\Gamma_i$. In the case 
$i=p$, the characteristic polynomial on $\Gamma_p$ is defined 
by $P_p(X)=1$ if $m_p=0$, and by
$$
    P_p(X)=  \sum_{(j,\alpha) \in \Lambda_0 \cap \Gamma_p}
           c_{j,\alpha}(0) X^{j}
        = c_{m_p,n_p}(0)X^{m_p}+ \cdots, 
          \quad\mbox{if $m_p \geq 1$}.
$$
In the case $m_p \geq 1$, the roots $\lambda_{p,h}$ 
($1 \leq h \leq m_p$) of $P_p(X)=0$ are called the 
characteristic roots on $\Gamma_p$. We define 
{\it the generalized Poincar\'e condition} in the 
following way:
\par
\medskip
   (GP)(Generalized Poincar\'e condition) 
\par
   \quad (i) $\lambda_{i,h} \in \BC \setminus [0,\infty)$
       for all $1 \leq i \leq p-1$ and 
       $1 \leq h \leq m_i-m_{i+1}$, 
\par
   \quad (ii) $\lambda_{p,h} \in \BC \setminus \BN^*$
        for $1 \leq h \leq m_p$.

\medskip
\noindent\textbf{Remark.} For $p=1$, we have ${\mathcal N}_0=\{(x,y) \in \BR^2 \,;\,
x \leq m, y \leq 0 \}$. Therefore, (GP) is reduced to its second statement, and (GP) is equivalent to (N).

\medskip
   We set
\begin{equation}\label{e22}
    \phi(\lambda,\rho)=\sum_{i=1}^p
             \lambda^{m_i} \rho^{n_i}
       = \lambda^m + \lambda^{m_2} \rho^{n_2}
           + \cdots + \lambda^{m_p} \rho^{n_p}. 
\end{equation}

The following results can be found in Proposition 1 and Theorem 2,~\cite{poincare}, respectively.

\begin{lemma}\label{lema22} The following two conditions are equivalent:
\begin{itemize}
\item (N) and (GP) hold.
\item There exists $c_0>0$ such that
\begin{equation}\label{e23}
     |L(k,l)| \geq c_0 \phi(k,l),
\end{equation}		
for every $(k,l) \in \BN^* \times \BN$.
\end{itemize}
\end{lemma}

\begin{theo}[Theorem 2 in~\cite{poincare}]\label{teo23}
If (N), (R) and (GP) hold, the unique formal power series 
solution in Proposition~\ref{prop11} is convergent in a neighborhood of
$(0,0) \in \BC_t \times \BC_x$.
\end{theo}

\subsection{On the irregularity at $x=0$}\label{sec23}
We set $\Lambda = \{(j,\alpha) \in I_m \,;\, c_{j,\alpha}(x) \not\equiv 0 \}$ and $\Lambda_1 = \{(j,\alpha) \in I_m \,;\, c_{j,\alpha}(0)=0, \,\, c_{j,\alpha}(x) \not\equiv 0 \}$. It holds that $\Lambda \cup \{(m,0)\}=\Lambda_0 \cup \Lambda_1$.
For $(j,\alpha) \in \Lambda_1$ we define
\begin{align*}
    &p_{j,\alpha}= \mbox{the order of the zeros 
              of $c_{j,\alpha}(x)$ at $x=0$} \quad (\geq 1), \\
    &d_{j,\alpha} 
      = \min \{y \in \BR \,;\, (j,\alpha-y) \in {\mathcal N}_0 \}.
\end{align*}
Observe that $p_{j,\alpha}\ge1$ and $d_{j,\alpha}$ is well-defined, for every $(j,\alpha)\in\Lambda_1$. Moreover, one has $(j,\alpha) \in {\mathcal N}_0$ if and only if
$d_{j,\alpha} \leq 0$. We define the {\it irregularity $\sigma_0$ at $x=0$} of (\ref{e11}) by
\begin{equation}\label{e24}
      \sigma_0 
       = \max \Bigl[ \, 1, \, \max_{(j,\alpha) \in \Lambda_1}
         \frac{p_{j,\alpha}+d_{j,\alpha}}{p_{j,\alpha}}
             \Bigr].   
\end{equation}
The reason why we call this ``the irregularity at $x=0$'' is explained by the following lemma:
\begin{lemma}\label{lema24} The regular singularity condition (R) (see Definition~\ref{defi21}) is satisfied if and only if $\sigma_0=1$.
\end{lemma}

\subsection{Main results}\label{sec24}
Given a formal power series $f(t,\bz)= \sum_{i+|\bnu| \geq 0} f_{i,\bnu} t^i \bz^{\bnu}\in \BC [[t,\bz]]$,
we define the {\it valuation $\mathrm{val}(f)$ of $f(t,\bz)$} by 
$$ \mathrm{val}(f)= \min \{ i+|\bnu| \,;\, f_{i,\bnu} \ne 0 \}.$$
If $f(t,\bz) \equiv 0$ we set $\mathrm{val}(f)=\infty$. The previous definition is naturally extended to a holomorphic function defined in a neighborhood of 
$(0,0) \in \BC_t \times \BC_{\bz}^N$ by means of its Taylor expansion at the origin.

Let $\sigma_0$ be the irregularity at $x=0$ of (\ref{e11}), and let
$R_2(t,x,\bz)$ be as in (\ref{e12}). We put
$$
    L_{\mu,j,\alpha} = \mathrm{val} \bigl(
          ( \partial_{z_{j,\alpha}}
           \partial_x^{\, \, \mu} R_2)(t,0,\bz) \bigr), 
       \quad  \mu \in \BN, \enskip 
                              (j,\alpha) \in I_m
$$
and set
\begin{equation}\label{e25}
     s_0= 1 + \max \biggl[ \, 0, \, \max_{0 \leq \mu<m}
          \Bigl( \max_{(j,\alpha) \in I_m}
              \frac{j+\mu+ \sigma_0(\alpha-\mu)-m}
                   {L_{\mu,j,\alpha}} \Bigr) 
           \biggr].
\end{equation}

\begin{theo}[Main Theorem]\label{maintheorem} Assume the conditions
$\mbox{A}_1)$, $\mbox{A}_2)$, $\mbox{A}_3)$, (N) and (GP) hold. Let 
$u(t,x) \in \BC[[t,x]]$ be the unique formal solution of (\ref{e11}) 
satisfying $u(0,x) \equiv 0$. Then, the following results hold:
\begin{itemize}
\item[(i)] If $\sigma_0 =1$, then $u(t,x)$ is convergent in a neighborhood of
$(0,0) \in \BC_t \times \BC_x$.
\item[(ii)] If $\sigma_0>1$, then $u(t,x) \in G\{t,x \}_{(s,\sigma)}$ 
for any $s \geq s_0$ and $\sigma \geq \sigma_0$.
\end{itemize}
\end{theo}

Since $\sigma_0=1$ is equivalent to condition (R) (see Lemma~\ref{lema24}), the first part in the previous result is a known fact, which can be found in Theorem~\ref{teo23}. In the case $m=1$, the 
second statement of the previous result was proved by H. Chen, Z. Luo and the second author~\cite{CLT}. The proof of such statement under general settings is put forward in Section~\ref{sec4}. It is worth mentioning that indices close to $\sigma_0$ and $s_0$ are defined in the work by H. Chen and Z. Luo~\cite{CL2}.

For $0 \leq \mu<m$ we set 
$$
     I_{m,\mu}= \bigl\{(j,\alpha) \in I_m \,;\, \alpha >\mu, 
      \, (\partial_{z_{j,\alpha}}
     \partial_x^{\, \, \mu}R_2)(t,0,\bz) \not\equiv 0 \bigr\}.
$$

The next result is a direct consequence of Theorem~\ref{maintheorem}.

\begin{corol}\label{coro26} Under the assumption
\begin{equation}\label{e26}
    \sigma_0 \leq \frac{m-j-\mu}{\alpha-\mu} \quad  
     \mbox{for $0 \leq \mu <m$ and $(j,\alpha) \in I_{m,\mu}$}
\end{equation}
one has $u(t,x) \in G\{t,x \}_{(1,\sigma_0)}$.
\end{corol}

Corollary~\ref{coro26} implies that the unique formal power series 
solution $u(t,x)$ of (\ref{e11}) is holomorphic in the variable $t$.

The following theorem asserts that our condition in Theorem~\ref{maintheorem} 
is optimal in a generic case.

\begin{theo}[Optimality]\label{teo27} Assume the conditions
$\mbox{A}_1)$, $\mbox{A}_2)$ and $\mbox{A}_3)$ hold. In addition to that, we adopt the following conditions:
\begin{itemize}
\item[$c_1)$] $(\partial_x^{\,\,\mu}a)(0)>0$ for $0 \leq \mu \leq m$ and $a(x) \gg 0$,
\item[$c_2)$] $c_{j,\alpha}(0) \leq 0$  for every $(j,\alpha) \in I_m$, 
\item[$c_3)$] $c_{j,\alpha}(x)-c_{j,\alpha}(0) \gg 0$ for every $(j,\alpha) \in I_m$, 
\item[$c_4)$] $a_{i,\bnu}(x) \gg 0$ for all $(i,\bnu) \in \BN \times \BN^N$ with $i+|\bnu| \geq 2$.
\end{itemize}
Then, equation (\ref{e11}) has a unique formal solution 
$u(t,x) \in \BC[[t,x]]$ satisfying $u(0,x) \equiv 0$. Moreover, $u(t,x)\in G\{t,x \}_{(s,\sigma)}$ if and only if $(s,\sigma)$ is such that $s \geq s_0$ and $\sigma \geq \sigma_0$.
\end{theo}

In view of the previous result, we may say that the index $(s_0,\sigma_0)$ defined in (\ref{e24}) 
and (\ref{e25}) is {\it the formal Gevrey index of the equation {\rm (\ref{e11})}}. The proof of Theorem~\ref{teo27} is given in detail in Section~\ref{secproofmaintheo}. 
\par

\vspace{3mm}

\noindent \textbf{Example:} We consider the equation
\begin{equation}\label{e27}
    ((t\partial_t)^4+(x \partial_x)^2)u
     =a(x)t+ x (t\partial_t)^2(x \partial_x)^2u
      +x^{\mu}t^i ((t\partial_t)^j\partial_x^{\alpha}u)^n,
\end{equation}
where $a(x) \in \BC\{x \}$, $(j,\alpha) \in I_4$, and 
$\mu,i,n \in \BN$ with $i+n \geq 2$ and $n \geq 1$. Suppose the 
conditions $(\partial_xa)(0)>0$,
$(\partial_x^{\, \alpha}a)(0)>0$ (only in the case $n \geq 2$), 
and $a(x) \gg 0$ hold. Then we have:
\begin{itemize}
\item $\sigma_0=2$ and $s_0= 1 + \max \left[ \, 0, \, 
    \frac{j+ 2\alpha-\mu-4}{i+n-1} \right].$
\item The equation (\ref{e27}) has a unique formal solution
$u(t,x) \in \BC[[t,x]]$ satisfying $u(0,x) \equiv 0$, and it 
belongs to the clsass $G\{t,x \}_{(s,\sigma)}$ if and only if
$s \geq s_0$ and $\sigma \geq 2$.
\item The formal solution $u(t,x)$ belongs to $G\{t,x \}_{(1,2)}$,
if and only if one of the following conditions 1)$\sim$5) are satisfied:
\begin{align*}
   {\rm 1)} &\enskip \mu \geq 4, \\
   {\rm 2)} &\enskip \mbox{$\mu=3$ and $\alpha \leq 3$}, \\
   {\rm 3)} &\enskip \mbox{$\mu=2$ and 
                $(j,\alpha) \in \{(k,\beta) \in I_4 \,;\,
                \beta \leq 2 \} \cup \{(0,3) \}$}, \\
   {\rm 4)} &\enskip \mbox{$\mu=1$ and 
                $(j,\alpha) \in \{(k,\beta) \in I_4 \,;\,
                \beta \leq 1 \} \cup \{(0,2), (1,2) \}$}, \\
   {\rm 5)} &\enskip \mbox{$\mu=0$ and 
                $(j,\alpha) \in \{(k,\beta) \in I_4 \,;\, 
         k+\beta \leq 2\}
             \cup \{(2,1), (3,0) \}$}.
\end{align*}
\end{itemize}

\section{Some preparatory discussions}\label{sec3}
In this section, we present some preparatory discussions which are needed in the proof of (ii) of Theorem~\ref{maintheorem}. 

%




Let $m \in \BN$, $\sigma \geq 1$ and $f(x)=\sum_{j \geq 0}f_jx^j \in \BC[[x]]$. In \cite{CLT},  the authors make use of the formal Borel operator ${\mathcal B}_{\sigma}$ defined by
$$
     {\mathcal B}_{\sigma}[f](x)
         = \sum_{j \geq 0} \frac{f_j}{j!^{\sigma-1}} x^j
$$
in order to achieve a Maillet-type result. In this paper, we need a refinement. For this purpose, we define the operator ${\mathcal B}_{\sigma}^{(m)}$ by
$$
   {\mathcal B}_{\sigma}^{(m)}[f](x)= f_0+f_1x+ \cdots+f_{m-1}x^{m-1}
               + \sum_{j \geq m}\frac{f_j}{(j-m)!^{\sigma-1}}x^j= \sum_{j \geq 0} \frac{f_j}{[j-m]_+!^{\sigma-1}} x^j.
$$

%
%
%

\begin{lemma}\label{lema32} Let $f(x), g(x) \in \BC[[x]]$. We also take $\sigma \geq 1$ 
and $m \in \BN$. The following statements hold:
\begin{align*}
   &{\mathcal B}_{\sigma}[|f|](x)=
    {\mathcal B}_{\sigma}^{(0)}[|f|](x) 
           \ll {\mathcal B}_{\sigma}^{(1)}[|f|](x)
           \ll {\mathcal B}_{\sigma}^{(2)}[|f|](x) \ll \ldots;\\
   &{\mathcal B}^{(m)}_{\sigma}[fg](x) \ll 
        {\mathcal B}^{(m)}_{\sigma}[|f|](x)
                   \times {\mathcal B}^{(m)}_{\sigma}[|g|](x);\\
   &{\mathcal B}^{(m)}_{\sigma}[x^kf](x)
      = x^k {\mathcal B}^{(m-k)}_{\sigma}[f](x) \quad
       \mbox{for $1 \leq k \leq m$}. 
\end{align*}
\end{lemma}

The proof of Lemma~\ref{lema32} is straightforward.

%
%
%

A Nagumo-like result is also derived, which will be useful in the sequel.
\begin{lemma}\label{lema33} Let $m \in \BN^*$ and $0<R \leq 1$.
Suppose that $f(x) \in \BC[[x]]$ satisfies
\begin{equation}\label{e35}
    {\mathcal B}^{(m)}_{\sigma}[|f|](x)
       \ll \frac{C}{(R-x)^a}
\end{equation}
for some $C>0$ and $a \geq 1$. Then, it holds that
$${\mathcal B}^{(m-1)}_{\sigma}[\partial_x |f|](x)
         \ll \frac{aC}{(R-x)^{a+1}},\qquad {\mathcal B}^{(m)}_{\sigma}[\partial_x |f|](x)
         \ll \frac{(a+\sigma)^{\sigma}e^{\sigma}C}
                  {(R-x)^{a+\sigma}}.$$
\end{lemma}
\begin{proof}
We write $f(x)=\sum_{j \geq 0}f_jx^j$. By the assumption (\ref{e35}) we have
$$
   \left\{ \begin{array}{ll}
     {\displaystyle |f_j| \leq \frac{C}{R^{a+j}} 
                     \frac{a(a+1) \cdots (a+j-1)}{j!}},
           &\mbox{if $0 \leq j \leq m-1$}, \\[11pt]
     {\displaystyle \frac{|f_j|}{(j-m)!^{\sigma-1}} \leq 
          \frac{C}{R^{a+j}} \frac{a(a+1) \cdots (a+j-1)}{j!}},
           &\mbox{if $j \geq m$}.
       \end{array}
                   \right.
$$
These estimates yield 
$$
{\mathcal B}^{(m-1)}_{\sigma}[\partial_x |f|](x) \ll  \sum_{j \geq 0} \frac{aC}{R^{a+1+j}} 
                 \frac{(a+1) \cdots (a+j)}{j!} x^j
             = \frac{aC}{(R-x)^{a+1}},
$$
which proves the first statement of Lemma~\ref{lema33}. The second follows from the next estimates:
\begin{align*}
   {\mathcal B}^{(m)}_{\sigma}[\partial_x |f|](x)&\ll \sum_{0 \leq j \leq m-1}(j+1)
                \frac{C}{R^{a+j+1}} 
                    \frac{a(a+1) \cdots (a+j)}{(j+1)!}x^j  \hfill \\
   & \qquad\qquad + \sum_{j \geq m} 
            \frac{(j+1)(j+1-m)!^{\sigma-1}}{(j-m)!^{\sigma-1}}
             \frac{C}{R^{a+j+1}} 
                     \frac{a(a+1) \cdots (a+j)}{(j+1)!} x^j \\
   & \ll \sum_{j \geq 0} \frac{(j+1)^{\sigma-1}C}{R^{a+j+1}} 
                 \frac{a(a+1) \cdots (a+j)}{j!} x^j.
\end{align*}
Here, we have used that $1 \leq (j+1)$ 
(for $0 \leq j \leq m-1$) and $(j+1-m) \leq (j+1)$ 
(for $j \geq m$).

Let $A=(a+\sigma)^{\sigma}e^{\sigma}$. Since
\begin{align*}
     \frac{(a+\sigma)^{\sigma}e^{\sigma}C}{(R-x)^{a+\sigma}}
     = \sum_{j \geq 0} \frac{AC}{R^{a+\sigma+j}}
         \frac{(a+\sigma)(a+\sigma+1) 
         \cdots (a+\sigma+j-1)}{j!} x^j,
\end{align*}
and $R^{a+\sigma+j} \leq R^{a+j+1}$, the proof is concluded after checking that
$$
     \frac{(j+1)^{\sigma-1} \Gamma(a+j+1) \Gamma(a+\sigma)}
     {\Gamma(a) \Gamma(a+\sigma+j)} \leq A. 
$$
We refer to the proof of Lemma 5 in~\cite{CLT} for a detailed demonstration of such estimate. 
%
%
%
\end{proof}
\begin{corol}\label{coro34} Let $m \in \BN^*$ and $0<R \leq 1$.
Suppose that $f(x) \in \BC[[x]]$ satisfies (\ref{e35}). Then, for
all $1 \leq \mu \leq m$ and $k \geq 1$ we have
\begin{align*}
    &{\mathcal B}^{(m-\mu)}_{\sigma}[\partial_x^{\mu}|f|](x)
        \ll \frac{a(a+1) \cdots(a+\mu-1)C}{(R-x)^{a+\mu}},\\
    &{\mathcal B}^{(m-\mu)}_{\sigma}[\partial_x^{k+\mu}|f|](x)
        \ll \frac{a(a+1) \cdots(a+\mu-1)A_{\mu,k}C}
                 {(R-x)^{a+\mu+k \sigma}},
\end{align*}
where $A_{\mu,k} = e^{k \sigma}\prod_{h=1}^k (a+\mu+h\sigma)^{\sigma}$.
\end{corol}

\subsection{On the Newton polygon}\label{sec33}
Let ${\mathcal N}_0$ be the Newton polygon associated to equation (\ref{e11}), and
let $\phi(\lambda,\rho)$ be as in (\ref{e22}). For $(j,\alpha) \in I_m$
with $(j,\alpha) \not\in {\mathcal N}_0$ we recall that
$$
     d_{j,\alpha}= \min_{(j,x) \in {\mathcal N}_0} |\alpha-x|
      = \min \{y \in \BR \,;\, (j,\alpha-y) \in {\mathcal N}_0 \}.
$$
%
%
\begin{prop}\label{prop35} Let $(j,\alpha) \in I_m$. The 
following results hold.
\begin{itemize}
\item[(i)] If $(j,\alpha) \in {\mathcal N}_0$, we have
$$  k^j l^{\alpha} \leq \phi(k,l), \quad (k,l) \in \BN^* \times \BN.$$
\item[(ii)] If $(j,\alpha) \not\in {\mathcal N}_0$, for $p \geq 1$
and $\sigma \geq 1+d_{j,\alpha}/p$ we have
$$ \frac{k^j (l-p)^{\alpha}}{\phi(k,l)} \leq \Bigl(\frac{l!}{(l-p)!} \Bigr)^{\sigma-1}, 
      \quad \hbox{ for }(k,l) \in \BN^* \times \BN\hbox{ with }l \geq p.$$
\end{itemize}
\end{prop}
\begin{proof}
Part (i) is proved in Lemma 7~\cite{poincare}. Part (ii) follows from Lemma~\ref{lema36} and Lemma~\ref{lema37} below.

\begin{lemma}\label{lema36}
 Let $m,n,p \in \BN$ with $m \geq 1$, $n \geq 0$ and $p \geq 1$. Suppose that $0 \leq j<m$ and $\alpha>n(m-j)/m$. Set $d=\alpha-n(m-j)/m$. Then, if $\sigma \geq (p+d)/p$ holds, we have
\begin{equation}\label{e313}
          \frac{k^j (l-p)^{\alpha}}{k^m+l^n}
           \leq \Bigl(\frac{l!}{(l-p)!} \Bigr)^{\sigma-1}, 
      \quad (k,l) \in \BN^* \times \BN\hbox{ with }l \geq p.
\end{equation}
\end{lemma}

The geometry described in Lemma~\ref{lema36} is illustrated in Figure~\ref{fig:a3}.
\begin{figure}[ht]
	\centering
		\includegraphics{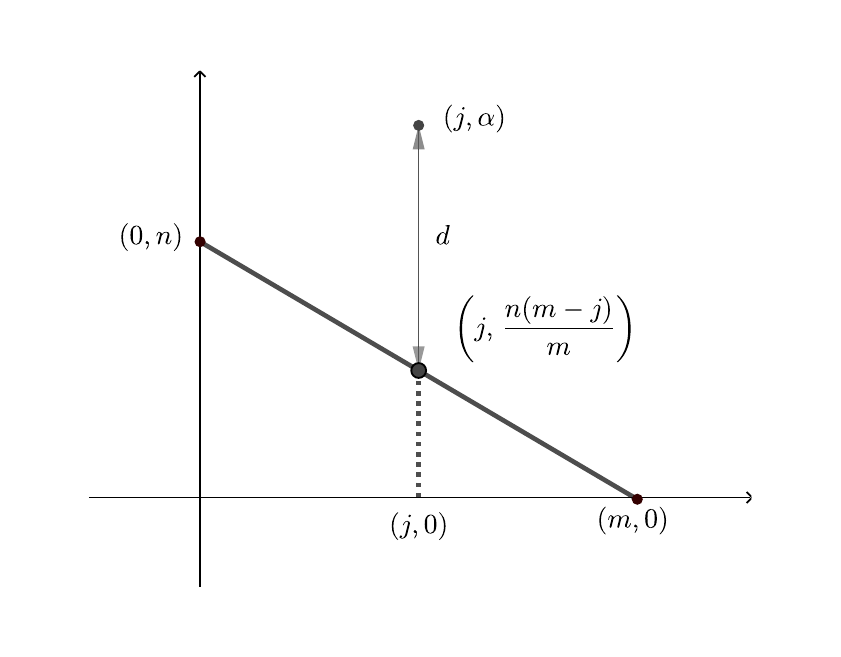}
	\caption{Geometry of the result described in Lemma~\ref{lema36}}
		\label{fig:a3}
\end{figure}
\begin{proof}

 If $j=0$ we have $d=\alpha-n>0$ and since $\sigma \geq (p+d)/p$ we have
$$
    \frac{k^j (l-p)^{\alpha}}{k^m+l^n}
        = \frac{(l-p)^{\alpha}}{k^m+l^n}
        \leq (l-p)^{\alpha-n}=(l-p)^d\le \Bigl(\frac{l!}{(l-p)!} \Bigr)^{\sigma-1},
$$
which yields (\ref{e313}).

If $n=0$, we have $d=\alpha$ and (\ref{e313}) follows from
$$
    \frac{k^j (l-p)^{\alpha}}{k^m+l^n}
        =  \frac{k^j(l-p)^{\alpha}}{k^m}
        \leq (l-p)^{\alpha} = (l-p)^d\le (l-p)^{p(\sigma-1)}.
$$

Let us consider the case $j>0$ and $n>0$. Let $a=m/j$ and $b=m/(m-j)$. Then, we have $1/a+1/b=1$
. From the application of Young's inequality we have
$$    k^j (l-p)^{n(m-j)/m}\leq \frac{1}{a}(k^j)^a + \frac{1}{b}((l-p)^{n(m-j)/m})^b = \frac{1}{a}k^m + \frac{1}{b}(l-p)^n    \leq k^m+l^n.$$
The result follows from $p(\sigma-1) \geq d$ and the fact that
$$  \frac{k^j (l-p)^{\alpha}}{k^m+l^n}= \frac{k^j (l-p)^{n(m-j)/m}}{k^m+l^n}(l-p)^{\alpha-n(m-j)/m}\leq  (l-p)^{\alpha-n(m-j)/m}=(l-p)^{d}.$$

\end{proof}

As a direct consequence of Lemma~\ref{lema36}, we have

\begin{lemma}\label{lema37} Let $m_i,m_{i+1},n_i,n_{i+1},p \in \BN$ with 
$m_i \geq 1$ and $p \geq 1$. Suppose that $m_{i+1} \leq j<m_i$, 
$n_{i+1} \geq n_i$ and 
$\alpha>n_i+ (n_{i+1}-n_i)(m_i-j)/(m_i-m_{i+1})$ hold.
Set $d=\alpha-n_i- (n_{i+1}-n_i)(m_i-j)/(m_i-m_{i+1})$. Then, if 
$\sigma \geq (p+d)/p$ holds we have
$$
    \frac{k^j (l-p)^{\alpha}}{k^{m_i}l^{n_i}+k^{m_{i+1}}l^{n_{i+1}}}
           \leq \Bigl(\frac{l!}{(l-p)!} \Bigr)^{\sigma-1}, 
      \quad (k,l) \in \BN^* \times \BN\hbox{ with }l \geq p.
$$
\end{lemma}
\begin{proof}
We apply Lemma~\ref{lema36} to 
$$
   \frac{k^j (l-p)^{\alpha}}
          {k^{m_i}l^{n_i}+k^{m_{i+1}}l^{n_{i+1}}}
   \leq  \frac{k^{j-m_{i+1}} (l-p)^{\alpha-n_i}}
                     {k^{m_i-m_{i+1}}+l^{n_{i+1}-n_i}}.
$$
\end{proof}



\end{proof}
\subsection{On an auxiliary equation}\label{sec34}

In this subsection, we consider the auxiliary equation
\begin{equation}\label{e314}
    C(x;k,x \partial_x)w= g(x) \in \BC[[x]]
\end{equation}
under the assumption $\sigma_0>1$, where $C(x;\lambda,\rho)$
is defined in (\ref{e13}). We note that the condition $\sigma_0>1$
is equivalent to the condition 
$\Lambda \setminus {\mathcal N}_0 \ne \emptyset$ (or 
$\Lambda_1 \setminus {\mathcal N}_0 \ne \emptyset$). 
We set $\Lambda_{out}=\Lambda \setminus {\mathcal N}_0$. If $\Lambda_{out} \ne \emptyset$ holds, the irregularity 
$\sigma_0$ is defined by
$$
      \sigma_0= \max_{(j,\alpha) \in \Lambda_{out}}
           \frac{p_{j,\alpha}+d_{j,\alpha}}{p_{j,\alpha}}.
$$
The definition of $p_{j,\alpha}$ described at the beginning of Section~\ref{sec23} allow us to  express the coefficients $c_{j,\alpha}(x)$ as follows:
$$
     c_{j,\alpha}(x)= \left\{
      \begin{array}{ll}
             x^{p_{j,\alpha}}b_{j,\alpha}(x), 
                    &\mbox{if $(j,\alpha) \in \Lambda_1$}, 
                    \\[3pt]
        c_{j,\alpha}(0)+ x^{p_{j,\alpha}}b_{j,\alpha}(x), 
    &\mbox{if $(j,\alpha) \in \Lambda_0 \setminus \{(m,0) \}$},
     \end{array}
                   \right.
$$
Observe that, in case $(j,\alpha) \in \Lambda_1$, the elements 
$p_{j,\alpha}$ are those in (\ref{e24}).

\begin{prop}\label{prop38}
Supppose the conditions (N), (GP) are taken for granted, and $\sigma_0>1$.  Then, for any $k \in \BN^*$ and $g(x) \in \BC[[x]]$ the equation (\ref{e314}) has a unique solution $w(x) \in \BC[[x]]$, and it holds that
\begin{equation}\label{e315}
    {\mathcal B}_{\sigma}^{(m)}[|w|](x)
    \ll \frac{A(x)}{k^m} {\mathcal B}_{\sigma}^{(m)}[|g|](x)
\end{equation}
for any $\sigma \geq \sigma_0$, where 
\begin{equation}\label{e899}
     A(x)= \frac{1}{c_0} \sum_{n \geq 0}
      \biggl( \frac{C_1}{c_0} \sum_{(j,\alpha) \in \Lambda} 
             x^{p_{j,\alpha}}
       {\mathcal B}_{\sigma}^{(m)}[|b_{j,\alpha}|](x) \biggr)^n,
\end{equation}
$c_0>0$ is the constant in (\ref{e23}), and $C_1>0$ is a constant which 
is independent of $k$ and $g(x)$.
\end{prop}

\begin{proof}
Take any $k \in \BN^*$ and $g(x) \in \BC[[x]]$.
For the sake of simplicity, we adopt the following notation: $[\rho]_0=1$;
$[\rho]_{\alpha}=\rho(\rho-1) \cdots (\rho-\alpha+1)$ (for
$\alpha \geq 1$). Then, equation (\ref{e314}) is written in the form
\begin{equation}\label{e316}
    L(k, x \partial_x)w=g(x)+ \sum_{(j,\alpha) \in \Lambda}
          x^{p_{j,\alpha}} b_{j,\alpha}(x) k^j 
         [x \partial_x]_{\alpha}w. 
\end{equation}
\par
   We set
$$
    w(x)= \sum_{l \geq 0}w_lx^l, \quad 
    g(x)= \sum_{l \geq 0}g_lx^l, \quad 
    b_{j,\alpha}(x)= \sum_{i \geq 0}b_{j,\alpha,i}x^i.
$$
Then, by substituting these series into (\ref{e316}) and 
comparing the coefficients of $x^l$ at both sides of the 
equation, this is decomposed into the following 
recurrence formulas:
$$
    L(k,l)w_l= g_l+ \sum_{(j,\alpha) \in \Lambda}
                 \sum_{i=0}^{l-p_{j,\alpha}}b_{j,\alpha,i}
          k^j [l-p_{j,\alpha}-i]_{\alpha}
             w_{l-p_{j,\alpha}-i}, \quad l \in \BN.
$$
Since $L(k,l) \ne 0$ for all $(k,l) \in \BN^* \times \BN$, $w_l$ is determined inductively for $l=0,1,2,\ldots$ Hence, equation (\ref{e316}) has a unique formal solution $w(x) \in \BC[[x]]$.

Let us show (\ref{e315}). By the assumptions (N), (GP), Lemma~\ref{lema22}
and Proposition~\ref{prop35} we have
\begin{align*}
    |w_l| &\leq \frac{1}{L(k,l)} \biggl( |g_l|
     + \sum_{(j,\alpha) \in \Lambda}
                 \sum_{i=0}^{l-p_{j,\alpha}}|b_{j,\alpha,i}|
          k^j [l-p_{j,\alpha}-i]_{\alpha}
             |w_{l-p_{j,\alpha}-i}| \biggr) \\
     &\leq \frac{1}{c_0 \phi(k,l)} \biggl( |g_l|
     + \sum_{(j,\alpha) \in \Lambda}
                 \sum_{i=0}^{l-p_{j,\alpha}}|b_{j,\alpha,i}|
          k^j (l-p_{j,\alpha})^{\alpha}
             |w_{l-p_{j,\alpha}-i}| \biggr) \\
    &\leq \frac{1}{c_0 k^m}|g_l|+ \frac{1}{c_0}
         \sum_{(j,\alpha) \in \Lambda}
                 \sum_{i=0}^{l-p_{j,\alpha}}|b_{j,\alpha,i}|
           \frac{l!^{\sigma-1}}{(l-p_{j,\alpha})!^{\sigma-1}}
             |w_{l-p_{j,\alpha}-i}| \\
    &\leq \frac{1}{c_0 k^m}|g_l|+ \frac{C_1}{c_0}
         \sum_{(j,\alpha) \in \Lambda}
                 \sum_{i=0}^{l-p_{j,\alpha}}|b_{j,\alpha,i}|
      \frac{[l-m]_+!^{\sigma-1}}{[l-p_{j,\alpha}-m]_+!^{\sigma-1}}
             |w_{l-p_{j,\alpha}-i}|
\end{align*}
for some constant $C_1>0$. Taking into account that
$$[i-m]_+!^{\sigma-1}[l-p_{j,\alpha}-i-m]_+!^{\sigma-1}\le [l-p_{j,\alpha}-m]_+!^{\sigma-1}$$
we conclude
$$
    {\mathcal B}_{\sigma}^{(m)}[|w|]
    \ll \frac{1}{c_0 k^m}{\mathcal B}_{\sigma}^{(m)}[|g|]
    + \frac{C_1}{c_0}\sum_{(j,\alpha) \in \Lambda} x^{p_{j,\alpha}}
      {\mathcal B}_{\sigma}^{(m)}[|b_{j,\alpha}|]
          \times {\mathcal B}_{\sigma}^{(m)}[|w|],
$$
which yields (\ref{e315}) by setting $A(x)$ as in (\ref{e899}).
\end{proof}

\section{Proof of (ii) of Theorem~\ref{maintheorem}}\label{sec4}

In this section, we give a proof of (ii) of Theorem~\ref{maintheorem} in the case $\sigma=\sigma_0$ and $s\ge s_0$, where $s_0$ is determined in (\ref{e25}). The first lemma provides a reformulation of the index $s_0$, which leans on the following construction. 

For $0 \leq \mu <m$ we define
$$
    J_{\mu}=\{(i,\bnu) \in \BN \times \BN^N \,;\,
         i+|\bnu| \geq 2, |\bnu| \geq 1, 
                 (\partial_x^{\mu}a_{i,\bnu})(0) \ne 0 \}.
$$
For $\mu \in \BN$ and 
$\bnu=\{\nu_{j,\alpha} \}_{(j,\alpha) \in I_m}$ satisfying 
$|\bnu| \geq 1$ we set
$$K_{\bnu}=\{(j,\alpha) \in I_m \,;\, \nu_{j,\alpha}>0 \}, \quad m_{\bnu,\mu} = \max_{(j,\alpha) \in K_{\bnu}}\bigl( j+ \max\{\alpha, \mu+\sigma_0(\alpha-\mu)\}\bigr).$$
If $\mu \geq m$ we have $m_{\bnu,\mu} \leq m$ for any $\bnu$
with $|\bnu| \geq 1$.  We have

\begin{lemma}\label{lema39} The index $s_0$ in (\ref{e25}) can be expressed
in the form
\begin{equation}\label{e317}
   s_0= 1+ \max \biggl[\,0 \,, \max_{0 \leq \mu<m}
            \Bigl( \sup_{(i,\bnu) \in J_{\mu}}
              \frac{m_{\bnu,\mu}-m}{i+|\bnu|-1} \Bigr)
             \biggr].
\end{equation}
\end{lemma}
\begin{proof} Set 
$f(\mu,j,\alpha)=j+\mu+ \sigma_0(\alpha-\mu)-m$: then $s_0$ is
given by (\ref{e25}) in the form
$$
    s_0 =1+\max \biggl[ \, 0, \, \max_{0 \leq \mu<m}
          \Bigl( \max_{(j,\alpha) \in I_m}
          \frac{f(\mu,j,\alpha)}{L_{\mu,j,\alpha}} 
          \Bigr) \biggr].
$$
Therefore, $s_0$ is determined only by $(\mu,j,\alpha)$ satisfying 
$f(\mu,j,\alpha)>0$.  Since 
$(\partial_{z_{j,\alpha}} \partial_x^{\, \, \mu}R_2)(t,0,\bz)$ 
is expressed in the form
$$
    ( \partial_{z_{j,\alpha}}
       \partial_x^{\, \, \mu}R_2)(t,0,\bz)
    = \sum_{(i,\bnu) \in J_{\mu}, \nu_{j,\alpha}>0}
            \nu_{j,\alpha}(\partial_x^{\mu}a_{i,\bnu})(0)
      t^i \bz^{\bnu-e_{j,\alpha}}
$$
(where $e_{j,\alpha} \in \BN^N$ is an $N$-vector defined by 
$\{\nu_{i,\beta} \}_{(i,\beta) \in I_m}$ with $\nu_{j,\alpha}=1$
and $\nu_{i,\beta}=0$ for $(i,\beta) \ne (j,\alpha)$), 
by the definition of $L_{\mu,j,\alpha}$ we have
\begin{align*}
    s_0 &=1+\max \biggl[ \, 0, \, \max_{0 \leq \mu<m}
          \Bigl( \max_{(j,\alpha) \in I_m} \Bigl(
            \sup_{(i,\bnu) \in J_{\mu}, \nu_{j,\alpha}>0}
              \,\frac{f(\mu,j,\alpha)}
                   {i+|\bnu|-1} \Bigr) \Bigr)
           \biggr] \\
    &=1+\max \biggl[ \, 0, \, \max_{0 \leq \mu<m}
          \Bigl( 
            \sup_{(i,\bnu) \in J_{\mu}} \Bigl(
              \max_{(j,\alpha) \in K_{\bnu}}
              \frac{f(\mu,j,\alpha)}
                   {i+|\bnu|-1} \Bigr) \Bigr)
           \biggr].
\end{align*}
We set $g(\mu,j,\alpha)= 
j+ \max\{\alpha, \mu+\sigma_0(\alpha-\mu)\}-m$. 
If $\alpha \leq \mu$ we have $f(\mu,j,\alpha) \leq 0$
and $g(\mu,j,\alpha)\leq 0$. If $\alpha > \mu$ we have
$g(\mu,j,\alpha)=f(\mu,j,\alpha)$.
Therefore, $s_0$ is determined only by $(\mu,j,\alpha)$ with
$\alpha>\mu$ and 
\begin{align*}
    s_0 
    &=1+\max \biggl[ \, 0, \, \max_{0 \leq \mu<m}
          \Bigl( 
            \sup_{(i,\bnu) \in J_{\mu}} \Bigl(
              \max_{(j,\alpha) \in K_{\bnu}}
              \frac{g(\mu,j,\alpha)}
                   {i+|\bnu|-1} \Bigr) \Bigr) \biggr].
\end{align*}
This proves (\ref{e317}). 
\end{proof}

Suppose the conditions (N), 
(GP) and $\sigma_0>1$ hold. Then, we have $\Lambda_{out} \ne \emptyset$. 
Let
$$
    u(t,x)= \sum_{k \geq 1} u_k(x) t^k 
        \in (\BC[[x]])[[t]]
$$
be the unique formal solution of (\ref{e11}).
Then, $u_k(x)$ ($k=1,2,\ldots$) are determined as the solutions
of the recurrence formulas:
$$
    C(x; k,x \partial_x)u_k = f_k(x), 
    \quad k=1,2,\ldots
$$
with $f_1(x)=a(x)$ and for $k \geq 2$
$$
    f_k(x) = \sum_{2 \leq i+|\bnu| \leq k}a_{i,\bnu}(x)
           \sum_{i+|k(\bnu)|=k}
        \prod_{(j,\alpha) \in I_m} 
        \prod_{h=1}^{\nu_{j,\alpha}} (k_{j,\alpha}(h))^j
               \partial_x^{\alpha}u_{k_{j,\alpha}(h)}, 
$$
where $\bnu=\{\nu_{j,\alpha} \}_{(j,\alpha) \in I_m}$
and $|k(\bnu)| =\sum_{(j,\alpha) \in I_m}(k_{j,\alpha}(1)
          + \cdots+k_{j,\alpha}(\nu_{j,\alpha}))$.
By Proposition~\ref{prop38} we have
\begin{equation}\label{e42}
    {\mathcal B}_{\sigma_0}^{(m)}[|u_k|](x)
    \ll \frac{A(x)}{k^m}
        {\mathcal B}_{\sigma_0}^{(m)}[|f_k|](x),
    \quad k=1,2,\ldots .
\end{equation}
   Since $f_1(x)$ is holomorphic at $x=0$, by (\ref{e42}) we see that
${\mathcal B}_{\sigma_0}^{(m)}[|u_1|](x)$ is holomorphic at $x=0$.
We can show by induction on $k$ that 
${\mathcal B}_{\sigma_0}^{(m)}[|u_k|](x)$ ($k \geq 1$) are all 
holomorphic at $x=0$.  Thus, we have that $u_k(x) \in G\{x \}_{\sigma_0}$ for all $k \geq 1$.

\subsection{On a majorant equation}

Let $0<R \leq 1$ be small enough so that $A(x) \in {\mathcal O}(\overline{D}_R)$, 
$a_{i,\bnu}(x) \in {\mathcal O}(\overline{D}_R)$ ($i+|\bnu| \geq 2$)
and ${\mathcal B}_{\sigma}^{(m)}[|u_1|](x) 
                  \in {\mathcal O}(\overline{D}_R)$. We take $A>0$ so that
\begin{equation}\label{e43}
     {\mathcal B}_{\sigma_0}^{(m-\mu)}[\partial_x^{\alpha}|u_1|](x)
      \ll \frac{A}{R-x}, \quad
       0\leq \mu \leq m, \, (j,\alpha) \in I_m,
\end{equation}
and $A_{i,\bnu} \geq 0$ ($i+|\bnu| \geq 2$) such that
\begin{equation}\label{e44}
     A(x) {\mathcal B}_{\sigma_0}^{(m)}[|a_{i,\bnu}|](x)
            \ll \frac{A_{i,\bnu}}{R-x}
     \quad \mbox{and} \quad
     \sum_{i+|\bnu| \geq 2} A_{i,\bnu}t^i Y^{|\bnu|} \in \BC\{t,Y \}.
\end{equation}
We take $L \in \BN^*$ so that $L \geq m\sigma_0$. Then we have
$L \geq j+\max\{\alpha, \mu+\sigma_0(\alpha-\mu) \}$ for any
$0 \leq \mu<m$ and $(j,\alpha) \in I_m$.
Set $H=(3m e \sigma_0)^{m\sigma_0}$. Under these notations, let
us consider the functional equation
\begin{equation}\label{e45}
    Y= \frac{A}{(R-x)^{m\sigma_0}}t + \frac{1}{(R-x)^{m\sigma_0}}
        \sum_{i+|\bnu| \geq 2}
         \frac{A_{i,\bnu}(i+|\bnu|)^L}
           {(R-x)^{m\sigma_0(3i+2|\bnu|-3)}} t^i (HY)^{|\bnu|}
\end{equation}
with respect to $(t,Y)$, where $x \in D_R$ is regarded as a 
parameter. By the implicit function theorem we see that for any 
$x \in D_R$ the equation (\ref{e45}) has a unique holomorphic solution 
$Y=Y(t)$ in a neighborhood of $t=0$ satisfying $Y(0)=0$. The coefficients of the Taylor expansion $Y=\sum_{k \geq 1}Y_k t^k$, are determined by the 
following recurrence formulas:
\begin{equation}\label{e46}
   Y_1= \frac{A}{(R-x)^{m\sigma_0}},
\end{equation}
and for $k \geq 2$
\begin{equation}\label{e47}
    Y_k  = \frac{1}{(R-x)^{m\sigma_0}}
           \sum_{2 \leq i+|\bnu| \leq k}
      \frac{A_{i,\bnu}(i+|\bnu|)^L}
           {(R-x)^{m\sigma_0(3i+2|\bnu|-3)}} \biggl[\sum_{i+|k(\bnu)|=k}
        \prod_{(j,\alpha) \in I_m} 
        \prod_{h=1}^{\nu_{j,\alpha}} HY_{k_{j,\alpha}(h)}
           \biggr].
\end{equation}
Moreover, by induction on $k$ we can show that $Y_k$ has the
form
$$
      Y_k= \frac{C_k}{(R-x)^{m\sigma_0(3k-2)}},
     \quad k=1,2,\ldots
$$
where $C_1=A$ and $C_k \geq 0$ ($k \geq 2$) are constants
which are independent of the parameter $x$.

\begin{lemma}\label{lema41}
Assume that $s \geq s_0$. Then, for any $k=1,2,\ldots$ we have
\begin{equation}\refstepcounter{equation}\label{e48}
    {\mathcal B}_{\sigma_0}^{(m-\mu)}
          [k^j\partial_x^{\alpha}|u_k|](x)
      \ll \frac{(k-1)!^{s-1}} 
   {k^{L-j-\max \{\alpha, \mu+\sigma_0(\alpha-\mu) \}}} H Y_k, \tag*{(\theequation)$_k$}
\end{equation}
for any $0\leq \mu \leq m$ and $(j,\alpha) \in I_m$.
\end{lemma}

\subsection{Proof of Lemma~\ref{lema41}}
\begin{proof}
In the case $k=1$, by (\ref{e43}) and (\ref{e46}) we have
$$
   {\mathcal B}_{\sigma_0}^{(m-\mu)}
      [1^j\partial_x^{\alpha}|u_1|](x)
      \ll \frac{A}{R-x}
      \ll \frac{A}{(R-x)^{m\sigma_0}}=Y_1 \ll HY_1
$$
for any $0\leq \mu \leq m$ and $(j,\alpha) \in I_m$. 
Hence we have \ref{e48} for $k=1$.  Let us show the general case 
by induction on $k$.

Let $k \geq 2$, and suppose that the equation is already proved for all $1 \leq p<k$. We express 
$$
    a_{i,\bnu}(x)= a_{i,\bnu,0}+a_{i,\bnu,1}x+
        \cdots+a_{i,\bnu,m-1}x^{m-1}+ x^m a_{i,\bnu,m}(x).
$$
Then, 
$$
    f_k(x) = \sum_{\mu=0}^m x^{\mu}
           \sum_{2 \leq i+|\bnu| \leq k}a_{i,\bnu,\mu}
           \sum_{i+|k(\bnu)|=k}
        \prod_{(j,\alpha) \in I_m} 
        \prod_{h=1}^{\nu_{j,\alpha}} (k_{j,\alpha}(h))^j
                  \partial_x^{\alpha}u_{k_{j,\alpha}(h)}
$$
and so by Lemma~\ref{lema32}, and setting ${\mathcal A}_{i,\bnu,\mu}=|a_{i,\bnu,\mu}|$
for $0 \leq \mu \leq m-1$, and 
${\mathcal A}_{i,\bnu,m}={\mathcal B}_{\sigma_0}[|a_{i,\bnu,m}|]$
we have 

$$
    {\mathcal B}^{(m)}_{\sigma_0}[|f_k|] 
    \ll  \sum_{\mu=0}^m x^{\mu}
           \sum_{2 \leq i+|\bnu| \leq k}
             {\mathcal A}_{i,\bnu,\mu}
           \biggl[
           \sum_{i+|k(\bnu)|=k}
        \prod_{(j,\alpha) \in I_m} 
        \prod_{h=1}^{\bnu_{j,\alpha}}
             {\mathcal B}^{(m-\mu)}_{\sigma_0}
             \bigl[(k_{j,\alpha}(h))^j
     \partial_x^{\alpha}u_{k_{j,\alpha}(h)} \bigr] \biggr].
$$
Thus, by (\ref{e42}), the definition of $m_{\bnu,\mu}$ and the induction hypothesis we have

\begin{equation}\label{e49}
   {\mathcal B}^{(m)}_{\sigma_0}[|u_k|]
      \ll \frac{A(x)}{k^m}\sum_{\mu=0}^m x^{\mu}
           \sum_{2 \leq i+|\bnu| \leq k}
           {\mathcal A}_{i,\bnu,\mu} \biggl[ 
           \sum_{i+|k(\bnu)|=k}
        \prod_{(j,\alpha) \in I_m} 
        \prod_{h=1}^{\nu_{j,\alpha}} 
        \frac{(k_{j,\alpha}(h)-1)!^{s-1}}{(k_{j,\alpha}(h)
         )^{L-m_{\bnu,\mu}}}
           HY_{k_{j,\alpha}(h)} \biggr].
\end{equation}

Observe the condition $L-m_{\bnu,\mu} \geq 0$ follows from the choice of $L$ so that
$L \geq \sigma_0m$.

\begin{lemma}\label{lema42}
Under the above situation, $|\bnu| \geq 1$ and 
${\mathcal A}_{i,\bnu,\mu} \ne 0$ 
(or ${\mathcal A}_{i,\bnu,m}(x) \not\equiv 0$) we have
\begin{align}
    &\frac{(k-i-|\bnu|)!^{s-1}}{k^{L+m-m_{\bnu,\mu}}}
      \leq \frac{(k-1)!^{s-1}(i+|\bnu|)^{[m_{\bnu,\mu}-m]_+}}{k^L}, \label{e410}\\
    &\frac{1}{k^m} \prod_{(j,\alpha) \in I_m} 
        \prod_{h=1}^{\nu_{j,\alpha}} 
        \frac{(k_{j,\alpha}(h)-1)!^{s-1}}
               {(k_{j,\alpha}(h) )^{L-m_{\bnu,\mu}}}
         \leq \frac{(k-1)!^{s-1}(i+|\bnu|)^L}{k^L}.  \label{e411}
\end{align}
\end{lemma}
\begin{proof} Let us show (\ref{e410}).  
If $m \geq m_{\bnu,\mu}$ we have $[m_{\bnu,\mu}-m]_+=0$ and so 
$$
   \frac{(k-i-|\bnu|)!^{s-1}}{k^{L+m-m_{\bnu,\mu}}}
   \leq \frac{(k-1)!^{s-1}}{k^L}
    = \frac{(k-1)!^{s-1}(i+|\bnu|)^{[m_{\bnu,\mu}-m]_+}}{k^L}.
$$
If $m_{\bnu,\mu}>m$, by the assumption $s \geq s_0$ and Lemma~\ref{lema39} we have $(i+|\bnu|-1)(s-1) \geq m_{\bnu,\mu}-m$
and so we have
\begin{align*}
  \frac{(k-i-|\bnu|)!^{s-1}}{k^{L+m-m_{\bnu,\mu}}} & \leq \frac{(k-1)!^{s-1}}{k^L} 
      \frac{k^{m_{\bnu,\mu}-m}}{(k-i-|\bnu|+1)^{m_{\bnu,\mu}-m}}\\
   &\le \frac{(k-1)!^{s-1}}{k^L} 
      \Bigl(1+ \frac{i+|\bnu|-1}{k-i-|\bnu|+1}
         \Bigr)^{m_{\bnu,\mu}-m}\le \frac{(k-1)!^{s-1}}{k^L} 
       ( i+|\bnu|)^{m_{\bnu,\mu}-m}.
\end{align*}
This proves (\ref{e410}). In order to prove (\ref{e411}), we note that, if $k_j \geq 1$ 
($j=1,\ldots, |\bnu|$) and $k_1+\cdots+k_{|\bnu|}=k-i$ 
hold, then we have $k_j \leq (k_1 \cdots k_{|\bnu|})$ for 
$j=1,\ldots,|\bnu|$ and so
$k-i=k_1+\cdots+k_{|\bnu|} 
     \leq |\bnu| (k_1 \cdots k_{|\bnu|})$ which yields 
$k \leq (i+|\bnu|)(k_1 \cdots k_{|\bnu|})$. Therefore, by the same argument we have
$$
    \prod_{(j,\alpha) \in I_m} 
        \prod_{h=1}^{\nu_{j,\alpha}} 
        \frac{1}{(k_{j,\alpha}(h) )^{L-m_{\bnu,\mu}}}
         \leq \Bigl( \frac{i+|\bnu|}{k} \Bigr)^{L-m_{\bnu,\mu}}.
$$
Hence, by the condition $i+|k(\bnu)|=k$ and (\ref{e410}) we have
\begin{align*}
   &\frac{1}{k^m} \prod_{(j,\alpha) \in I_m} 
        \prod_{h=1}^{\nu_{j,\alpha}} 
        \frac{(k_{j,\alpha}(h)-1)!^{s-1}}
               {(k_{j,\alpha}(h) )^{L-m_{\bnu,\mu}}} \leq \frac{(|k(\bnu)|-|\bnu|)!^{s-1}}{k^m} 
          \prod_{(j,\alpha) \in I_m} 
        \prod_{h=1}^{\nu_{j,\alpha}} 
        \frac{1}{(k_{j,\alpha}(h) )^{L-m_{\bnu,\mu}}} \\
   &\qquad \leq \frac{(k-i-|\bnu|)!^{s-1}}{k^m} \times 
          \frac{(i+|\bnu|)^{L-m_{\bnu,\mu}}}{k^{L-m_{\bnu,\mu}}} \leq \frac{(k-1)!^{s-1}}{k^L} 
       (i+|\bnu|)^{[m_{\bnu,\mu}-m]_+} 
                   \times (i+|\bnu|)^{L-m_{\bnu,\mu}}
\end{align*}
which proves (\ref{e411}). 
\end{proof}

By applying (\ref{e411}) to \ref{e48} we have
\begin{equation}\label{e412}
   {\mathcal B}^{(m)}_{\sigma_0}[|u_k|]
   \ll \frac{(k-1)!^{s-1}}{k^L} A(x) \sum_{\mu=0}^m x^{\mu}
           \sum_{2 \leq i+|\bnu| \leq k}
            {\mathcal A}_{i,\bnu,\mu}(i+|\bnu|)^L
         \biggl[\sum_{i+|k(\bnu)|=k}
        \prod_{(j,\alpha) \in I_m} 
        \prod_{h=1}^{\nu_{j,\alpha}}HY_{k_{j,\alpha}(h)} \biggr].
\end{equation}
By the definition of ${\mathcal A}_{i,\bnu,\mu}$ 
($0 \leq \mu \leq m$) we have $\sum_{\mu=0}^m x^{\mu}{\mathcal A}_{i,\bnu,\mu} = {\mathcal B}^{(m)}_{\sigma_0}[|a_{i,\bnu}|](x)$, and by (\ref{e44}) we have
$$
    A(x)\sum_{\mu=0}^m x^{\mu}{\mathcal A}_{i,\bnu,\mu}
    = A(x){\mathcal B}^{(m)}_{\sigma_0}[|a_{i,\bnu}|](x)
    \ll \frac{A_{i,\bnu}}{R-x}.
$$
By applying this to (\ref{e412}), and by (\ref{e47}) we derive
\begin{equation}\label{e413}
    {\mathcal B}^{(m)}_{\sigma_0}[|u_k|]\ll \frac{(k-1)!^{s-1}}{k^L} (R-x)^{m\sigma_0}Y_k =\frac{(k-1)!^{s-1}}{k^L}\frac{C_k}{(R-x)^{m\sigma_0(3k-3)}}.
\end{equation}

If $\alpha \leq \mu$, by Lemma~\ref{lema32}, (\ref{e413}) and Corollary~\ref{coro34} we have
\begin{align}
    {\mathcal B}^{(m-\mu)}_{\sigma_0}
                      [k^j\partial_x^{\alpha}|u_k|]
    &\ll k^j{\mathcal B}^{(m-\alpha)}_{\sigma_0}
                         [\partial_x^{\alpha}|u_k|] \ll \frac{(k-1)!^{s-1}}{k^{L-j}} 
     \frac{\prod_{i=0}^{\alpha-1}(m\sigma_0(3k-3)+i) \times C_k}
              {(R-x)^{m\sigma_0(3k-3)+\alpha}} \label{e414}\\
    &\ll \frac{(k-1)!^{s-1}}{k^{L-j-\alpha}} 
     \frac{(3m\sigma_0)^{\alpha}  C_k}
              {(R-x)^{m\sigma_0(3k-2)}} \ll \frac{(k-1)!^{s-1}}{k^{L-j-\alpha}}HY_k(x)\nonumber.
\end{align}
If $\mu<\alpha$, the application of (\ref{e413}) and 
Corollary~\ref{coro34} yield
$$ {\mathcal B}^{(m-\mu)}_{\sigma_0}
                      [k^j\partial_x^{\alpha}|u_k|]
    = k^j{\mathcal B}^{(m-\mu)}_{\sigma_0}
               [\partial_x^{(\alpha-\mu)+\mu}|u_k|] \ll \frac{(k-1)!^{s-1}}{k^{L-j}} 
     \frac{A(\mu,\alpha)  C_k}
        {(R-x)^{m\sigma_0(3k-3)+\mu+\sigma_0(\alpha-\mu)}}$$
where
$$
     A(\mu,\alpha)= \prod_{i=0}^{\mu-1}(m\sigma_0(3k-3)+i) 
         \prod_{h=1}^{\alpha-\mu} 
        \bigl[(m\sigma_0(3k-3)+\mu+h \sigma_0)^{\sigma_0}
          e^{\sigma_0} \bigr].
$$
Since $A(\mu,\alpha) 
       \leq k^{\mu+\sigma_0(\alpha-\mu)}(3m\sigma_0)^{\sigma_0\alpha}
            e^{\sigma_0(\alpha-\mu)}$, we have
\begin{align}
    {\mathcal B}^{(m-\mu)}_{\sigma_0}
                      [k^j\partial_x^{\alpha}|u_k|]
    &\ll \frac{(k-1)!^{s-1}}{k^{L-j-(\mu+\sigma_0(\alpha-\mu))}} 
     \frac{(3m\sigma_0)^{\sigma_0\alpha}e^{\sigma_0(\alpha-\mu)} C_k}
        {(R-x)^{m\sigma_0(3k-3)+\mu+\sigma_0(\alpha-\mu)}}
         \label{e415} \\
    &\ll \frac{(k-1)!^{s-1}}{k^{L-j-(\mu+\sigma_0(\alpha-\mu))}} 
     \frac{H  C_k}
        {(R-x)^{m\sigma_0(3k-3)+m \sigma_0}} = \frac{(k-1)!^{s-1}}{k^{L-j-(\mu+\sigma_0(\alpha-\mu))}}
             HY_{k}.\nonumber
\end{align}
By (\ref{e414}) and (\ref{e415}) we have \ref{e48}. This proves Lemma~\ref{lema41}. 

\end{proof}

\subsection{Completion of the proof of (ii) of Theorem~\ref{maintheorem}}\label{sec43}

   By Lemma~\ref{lema41} we have
\begin{align*}
   \sum_{k \geq 1}
      \frac{{\mathcal B}_{\sigma_0}[|u_k|](x)}{(k-1)!^{s-1}} t^k
   \ll \sum_{k \geq 1}
    \frac{{\mathcal B}^{(m)}_{\sigma_0}
                   [|u_k|](x)}{(k-1)!^{s-1}} t^k
   \ll \sum_{k \geq 1}\frac{1}{k^L} HY_k(x) t^k.
\end{align*}
Take any $r \in (0,R)$.  We know there
is $\delta>0$ such that 
$\sum_{k \geq 1}Y_k(r)t^k$ is convergent for $|t| \leq \delta$.
Then, for $|t| \leq \delta$ we have
\begin{align*}
   \sum_{k \geq 1} {\mathcal B}_{\sigma_0}
           [|u_k|](r)\frac{|t|^k}{(k-1)!^{s-1}}
     \leq H \sum_{k \geq 1} Y_k(r) \delta^k <\infty.
\end{align*}
This proves that $u(t,x) \in G\{t,x \}_{(s,\sigma_0)}$ holds,
and we have (ii) of Theorem~\ref{maintheorem}.  \qed

\section{Proof of Theorem~\ref{teo27}}\label{secproofmaintheo}

    Suppose the conditions A${}_1$)$\sim$A${}_3$) and
c${}_1$)$\sim$c${}_4$) hold. Since $L(\lambda,\rho)$ is defined by
$$
     L(\lambda,\rho)=\lambda^m + \sum_{(j,\alpha) \in 
                (\Lambda_0 \setminus \{(m,0)\})}
                (-c_{j,\alpha}(0))\lambda^j [\rho]_{\alpha}
$$
(where $[\rho]_0=1$ and 
$[\rho]_{\alpha}=\rho(\rho-1) \cdots (\rho-\alpha+1)$ for
$\alpha \geq 1$) and since $-c_{j,\alpha}(0)>0$ holds for any 
$(j,\alpha) \in (\Lambda_0 \setminus \{(m,0)\})$, we have
$L(k,l) \geq k^m \geq 1$ for any $(k,l) \in \BN^* \times \BN$.
This means that the condition (N) is satisfied which entails that
the equation (\ref{e11}) has a unique formal solution 
$u(t,x) \in \BC[[t,x]]$ satisfying $u(0,x) \equiv 0$.
\par
   Since $(m_i,n_i) \in \Lambda_0$ for $i=1,\ldots,p$ and since
the coefficients of $\lambda^{m_i}[\rho]_{n_i}$ 
($i=1,\ldots,p$) in $L(\lambda, \rho)$ are all positive, we have 
$L(k,l) \geq c_0 \phi(k,l)$ on 
$\{(k,l) \in \BN^* \times \BN \,;\, l \geq m \}$ for some $c_0>0$. 
Since $m^{n_i} \geq l^{n_i}$ for $0 \leq l \leq m-1$, by setting $\delta_i=1/m^{n_i}$ we have 
$L(k,l) \geq k^m \geq \delta_ik^{m_i}l^{n_i}$ on 
$\{(k,l) \in \BN^* \times \BN \,;\, l<m \}$.
Hence, we have $L(k,l) \geq c_1 \phi(k,l)$ on 
$\{(k,l) \in \BN^* \times \BN \,;\, l<m \}$ for some $c_1>0$.
Thus, by Lemma~\ref{lema22} we see that the generalized Poincar\'e condition 
(GP) is satisfied. Hence, by Theorem~\ref{maintheorem} we have 
$u(t,x) \in G\{t,x \}_{(s,\sigma)}$ provided that
$(s,\sigma)$ satisfies $s \geq s_0$ and $\sigma \geq \sigma_0$.

\subsection{Proof of the converse statement}

Suppose that 
$u(t,x) \in G\{t,x \}_{(s,\sigma)}$ holds for some $s \geq 1$ and
$\sigma \geq 1$. Since $s_0$ is expressed as (\ref{e317}) and since
$$j+\max\{\alpha, \mu+\sigma_0(\alpha-\mu) \}-m
    = \left\{ \begin{array}{ll}
      j+\alpha+(\sigma_0-1)(\alpha-\mu)-m, 
                         &\mbox{if $\alpha>\mu$}, \\
      j+\alpha-m \leq 0, &\mbox{if $\alpha \leq \mu$},
    \end{array}
               \right.
$$
in order to show the conditions $s \geq s_0$ and 
$\sigma \geq \sigma_0$, it is enough to prove that the two 
conditions
\begin{equation}\label{e51}
\sigma \geq \frac{p_{h,\beta}+d_{h,\beta}}{p_{h,\beta}},\quad 
        s-1 \geq \frac{j+\alpha+(d_{h,\beta}/p_{h,\beta})
               (\alpha-\mu)-m}{i+|\bnu|-1} 
\end{equation}
hold for any $(h,\beta) \in \Lambda_{out}$, 
$\mu \in \{0,1,\ldots,m-1 \}$, $(i,\bnu) \in J_{\mu}$ and 
$(j,\alpha) \in K_{\bnu}$ satisfying $\alpha>\mu$. Let us show
this now. Take any $(h,\beta) \in \Lambda_{out}$, 
$\mu \in \{0,1,\ldots,m-1 \}$, $(i,\bnu) \in J_{\mu}$ and 
$(j,\alpha) \in K_{\bnu}$ with $\alpha >\mu$. Note that equation (\ref{e11}) is written as
$$    L(t \partial_t, x \partial_x)u 
    = a(x)t + \sum_{(j,\alpha) \in \Lambda}
    x^{p_{j,\alpha}}b_{j,\alpha}(x)
           (t \partial_t)^j[x \partial_x]_{\alpha}u +\sum_{i+|\bnu| \geq 2} a_{i,\bnu}(x)t^i 
          \prod_{(j,\alpha) \in I_m}
                 \bigl[(t \partial_t)^j 
      \partial_x^{\alpha}u \bigr]^{\nu_{j,\alpha}},
$$
and that its formal solution 
$u(t,x)= \sum_{k \geq 1}u_k(x)t^k \in \BC[[t,x]]$ satisfies that
$u(t,x) \gg 0$ and $L(1,x\partial_x)u_1(x)=a(x)$.
Since $\partial_x^lu(t,x) \gg (\partial_x^lu_1)(0)t  = (\partial_x^la)(0)t/L(1,l)$ for any $l \in \BN$, we have
\begin{align*}
   L(t \partial_t, x \partial_x)u
   &\gg \frac{(\partial_x^ma)(0)}{m!}x^mt 
          + b_{h,\beta}(0)x^{p_{h,\beta}}
          (t \partial_t)^h[x \partial_x]_{\beta}u  \\
   &+ \frac{(\partial_x^{\mu}a_{i,\bnu})(0)}
                  {\mu !} x^{\mu}t^{i+|\bnu|-1}
      \prod_{(k,\gamma) \ne (j,\alpha)}
           \Bigl(\frac{(\partial^{\gamma}a)(0)}
           {L(1,\gamma)} \Bigr)^{\nu_{k,\gamma}} \times  \Bigl(\frac{(\partial^{\alpha}a)(0)}
                         {L(1,\alpha)} \Bigr)^{\nu_{j,\alpha}-1}
          \times (t \partial_t)^j\partial_x^{\alpha}u.
\end{align*}
Thus, by setting
$$A= \frac{(\partial_x^ma)(0)}{m!}, \quad
   B=b_{h,\beta}(0), \quad
   C=\frac{(\partial_x^{\mu}a_{i,\bnu})(0)}{\mu !}
      \prod_{(k,\gamma) \ne (j,\alpha)}
           \Bigl(\frac{(\partial^{\gamma}a)(0)}
           {L(1,\gamma)} \Bigr)^{\nu_{k,\gamma}} 
          \times \Bigl(\frac{(\partial^{\alpha}a)(0)}
                         {L(1,\alpha)} \Bigr)^{\nu_{j,\alpha}-1},
$$
we have $A>0$, $B>0$, $C>0$ and
\begin{equation}\label{e53}
   L(t \partial_t, x \partial_x)u \gg A x^mt  + Bx^{p_{h,\beta}}
          (t \partial_t)^h[x \partial_x]_{\beta}u  + C x^{\mu}t^{i+|\bnu|-1}(t \partial_t)^j\partial_x^{\alpha}u.
\end{equation}

Now, let us consider the equation
\begin{equation}\label{e54}
   L(t \partial_t, x \partial_x)w=Ax^mt + Bx^{p_{h,\beta}}
          (t \partial_t)^h [x \partial_x]_{\beta}w  +C x^{\mu}t^{i+|\bnu|-1}
                (t \partial_t)^j\partial_x^{\alpha}w.
\end{equation}

\begin{lemma}\label{lema51} Under the above situation, the equation (\ref{e54})
has a unique formal solution $w(t,x) \in \BC[[t,x]]$ satisfying
$w(0,x) \equiv 0$, and it belongs to the class
$G\{t,x \}_{(s',\sigma')}$ if and only if $(s',\sigma')$ 
satisfies 
$$\sigma' \geq \frac{p_{h,\beta}+d_{h,\beta}}{p_{h,\beta}}, \quad s'-1 \geq \frac{j+\alpha+(d_{h,\beta}/p_{h,\beta})(\alpha-\mu)-m}{i+|\bnu|-1}.$$
\end{lemma}
The proof of this lemma will be given in Section~\ref{sec53}.
\par
   By (\ref{e53}) and (\ref{e54}), it holds that $u(t,x) \gg w(t,x)$. Since $u(t,x) \in G\{t,x \}_{(s,\sigma)}$ is assumed, we have
$w(t,x) \in G\{t,x\}_{(s,\sigma)}$, and so by Lemma~\ref{lema51} we have the conditions (\ref{e51}).

Thus, to complete the proof of Theorem~\ref{teo27} it is enough 
to show Lemma~\ref{lema51} above.

\subsection{Some lemmas}

Before the proof of Lemma~\ref{lema51}, let us give some lemmas which
are needed in that proof. We note that by the assumption
c${}_2$) we have $L(k,l) \geq k^m \geq 1$ for any 
$(k,l) \in \BN^* \times \BN$.

\begin{lemma}\label{lema52}
The following statements hold:
\begin{itemize}
\item[(i)] There is a constant $c_1>0$ such that $L(k,l) \leq c_1 \phi(k,l)$ for every $(k,l) \in \BN^* \times \BN$.
\item[(ii)] Let $a>0$ and $q \in \BN^*$. Then, there is $c_2>0$ 
with $L(kq+1,a) \leq c_2(k+1)^m$ for all $k \in \BN^*$.
\item[(iii)] Let $1 \leq i \leq p$, and let $-s_i$ be the slope of $\Gamma_i$. Then, there is a constant $c_3>0$ such that $\phi(k,l) \leq c_3 k^{m_i}l^{n_i}$ for every $l \in \BN^*$ and $k=[\,l^{s_i}]$.
\end{itemize}
\end{lemma}
\begin{proof}
The first part is a consequence of (i) of Proposition~\ref{prop35}, for
$c_1= 1+ \sum_{(j,\alpha) 
                    \in (\Lambda_0 \setminus \{(m,0) \})}
              |c_{j,\alpha}|$.
							
The statement (ii) is a consequence of the fact that $L(\lambda,a)$ is a polynomial of degree $m$ in $\lambda$.

In the case $1 \leq i<p$, the statement (iii) follows from Lemmas~\ref{lema53} and~\ref{lema54} given below. In the case $i=p$, then $s_p=0$ and $k=1$, so
$\phi(k,l)=\phi(1,l) \leq c_3 l^{n_p}=c_3 k^{m_p}l^{n_p}$ for some
$c_3>0$ (which is independent of $l$). 
\end{proof}

Both situations described in Lemma~\ref{lema53} and~\ref{lema54} are illustrated in Figure~\ref{fig:a34}.
\begin{figure}[ht]
	\centering
		\includegraphics[width=0.48\textwidth]{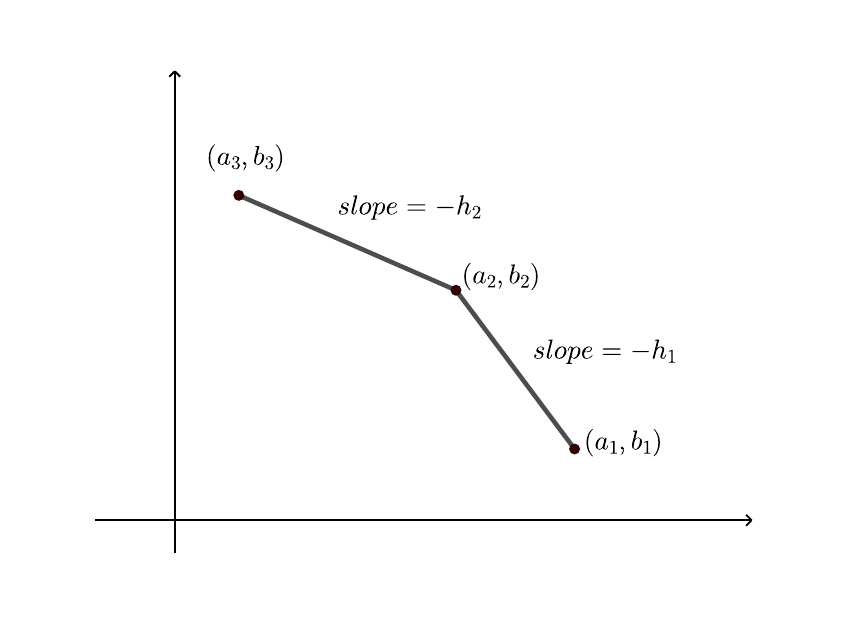}
		\includegraphics[width=0.48\textwidth]{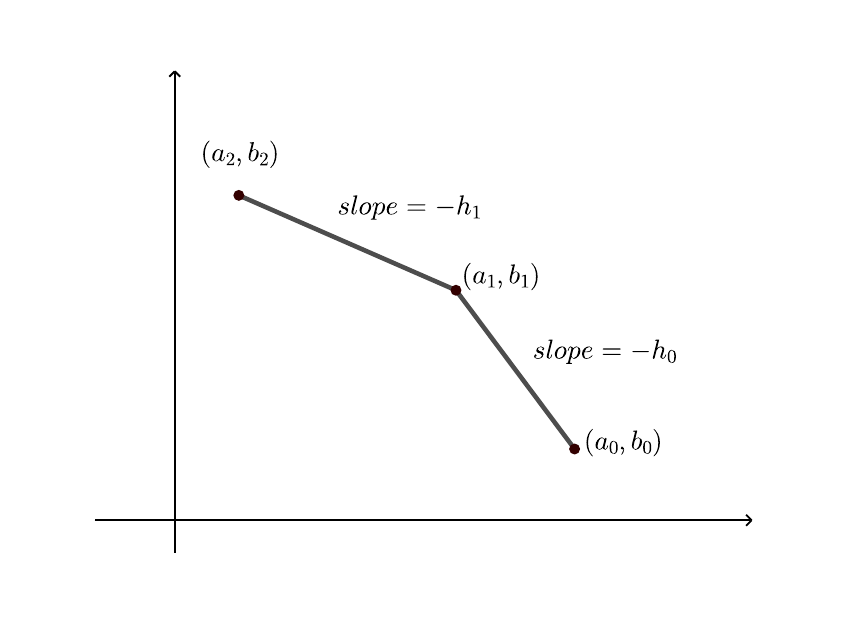} 
	\caption{Geometry in Lemma~\ref{lema53} (left) and Lemma~\ref{lema54} (right)}
			\label{fig:a34}
\end{figure}
\begin{lemma}\label{lema53} Let $0 \leq a_3<a_2<a_1$ and 
$0 \leq b_1<b_2<b_3$. Set $h_1=(b_2-b_1)/(a_1-a_2)$ and 
$h_2=(b_3-b_2)/(a_2-a_3)$.  If $h_1>h_2$ holds, there is 
$c>0$ such that
\begin{equation}\label{e55}
   k^{a_1}l^{b_1}+ k^{a_2}l^{b_2}+k^{a_3}l^{b_3}
     \leq c k^{a_1}l^{b_1}\quad 
    \mbox{for any $l \in \BN^*$ and $k=[l^{h_1}]$}.
\end{equation}
\end{lemma}
\begin{proof}
We set $h_*=(b_3-b_1)/(a_1-a_3)$. Then, $h_1>h_*$. If we fix $k=[l^{h_1}]$ we have
$$k^{a_1}l^{b_1}+ k^{a_2}l^{b_2}+k^{a_3}l^{b_3}\leq (l^{h_1})^{a_1}l^{b_1}+ (l^{h_1})^{a_2}l^{b_2}
               +(l^{h_1})^{a_3}l^{b_3} =(l^{h_1})^{a_1}l^{b_1} \Bigl( 1
          +1+ l^{-(h_1-h_*)(a_1-a_3)} \Bigr),$$
for every $l \in \BN^*$. Since $(h_1-h_*)(a_1-a_3)>0$, this leads us to (\ref{e55}).
\end{proof}
The proof of Lemma~\ref{lema54} is analogous to that of Lemma~\ref{lema53}, so we omit it.
\begin{lemma}\label{lema54}
Let $0 \leq a_2<a_1<a_0$ and $0 \leq b_0<b_1<b_2$. Set $h_0=(b_1-b_0)/(a_0-a_1)$ and 
$h_1=(b_2-b_1)/(a_1-a_2)$.  If $h_0>h_1$ holds, there is a constant $c>0$ such that
$$ k^{a_0}l^{b_0}+ k^{a_1}l^{b_1}+k^{a_2}l^{b_2} \leq c k^{a_1}l^{b_1}$$
for every $l \in \BN^*$ and $k=[l^{h_1}]$.
\end{lemma}
%
%
%
%
%
%
%

\begin{lemma}\label{lema55}
The following statements hold:
\begin{itemize}
\item[(i)] For any $a>0$, $b>0$, $c>0$, $d>0$ 
and $0 \leq \delta<1$ we have
$\displaystyle\lim_{\BN^* \ni l \to \infty}
          \frac{l!^a}{[cl^{\delta}+d]!^b}=\infty$.
\item[(ii)] For $a>b \geq 0$, $c>0$ and $0 \leq \delta<1$ we have
$\displaystyle \lim_{\BN^* \ni l \to \infty}
          \frac{l!^{a}}{[cl^{\delta}+ l]!^b}
     = \infty$.
\end{itemize}
\end{lemma}
\begin{proof}
Since 
$[cl^{\delta}+d] \leq cl^{\delta}+d \leq (c+d)l^{\delta}$ holds, 
to show (i) it is enough to prove
$$ \lim_{x \to \infty} 
 \frac{\Gamma(x+1)^a}{\Gamma(c_1x^{\delta}+1)^b}=\infty, \quad
    \mbox{that is,} \quad 
     \lim_{x \to \infty} \log \Bigl(
      \frac{\Gamma(x+1)^a}{\Gamma(c_1x^{\delta}+1)^b} \Bigr)
         = \infty,  
$$
where $c_1=c+d$. This is a direct consequence of Stirling's formula. The second part of the proof is attained by analogous arguments. 

\end{proof}

\subsection{Proof of Lemma~\ref{lema51}}\label{sec53}

Let $L(\lambda,\rho)$ be as in (\ref{e14}). By setting $p=p_{h,\beta}$ 
and $q=i+|\bnu|-1$, we can write the equation (\ref{e54}) as follows:
\begin{equation}\label{e58}
    L(t\partial_t,x \partial_x)u
     = Ax^mt + Bx^p(t\partial_t)^h[x\partial_x]_{\beta}u
          + Ct^qx^{\mu}(t\partial_t)^j \partial_x^{\alpha}u
\end{equation}
where $[x \partial_x]_0=1$
and 
$[x \partial_x]_{\beta}=(x \partial_x)
    (x \partial_x-1) \cdots(x \partial_x-\beta+1)$
for $\beta \geq 1$. 
For the sake of clarity, we summarize the main hypotheses on (\ref{e58}):
\begin{multicols}{2}
\begin{itemize}
\item[h${}_1$)] $A>0$, $B>0$, $C>0$;
\item[h${}_2$)] $p,q, \mu \in \BN$, with $p \geq 1$,  $q \geq 1$, 
         $0 \leq \mu<m$;
\item[h${}_3$)] $(h,\beta) \in I_m$ and 
             $(h,\beta) \not\in {\mathcal N}_0$;
\item[h${}_4$)]
         $(j,\alpha) \in I_m$ and $\alpha>\mu$.
\end{itemize}
\end{multicols}
Since $(h,\beta) \in I_m$, we have $0 \leq h<m$ and so we can find an $i \in \{1,\ldots,p \}$ such that $m_{i+1} \leq h<m_i$
holds. We set $d=d_{h,\beta}$. Then, $d= \min\{y \in \BR \,;\, (h,\beta-y) \in {\mathcal N}_0 \}
         =  \beta-n_i-s_i(m_i-h)$.
Since $(h,\beta) \not\in {\mathcal N}_0$ we have $d>0$. 
The situation is illustrated in Figure~\ref{e59}.
\begin{figure}[ht]
	\centering
		\includegraphics{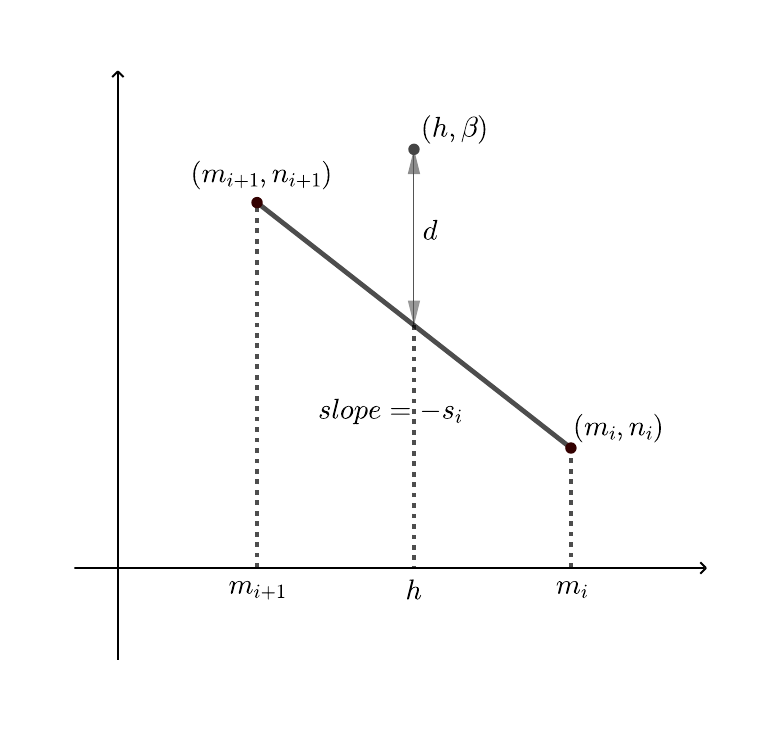}
	\caption{Geometry related to equation (\ref{e58})}
			\label{e59}
\end{figure}
Since $(h,\beta) \in I_m$ and $(h,\beta) \not\in {\mathcal N}_0$
we have $0 \leq s_i<1$. We set
$$\sigma_0^*= 1+\frac{d}{p}, \qquad
    s_0^* = 1+ \max \Bigl[ \, 0,\, 
        \frac{j+\alpha+(d/p)(\alpha-\mu)-m}{q} \Bigr].
$$
Then, Lemma~\ref{lema51} is stated in the following form:

\begin{prop}\label{prop56} The equation (\ref{e58}) has a unique
formal solution $u(t,x) \in \BC[[t,x]]$ satisfying 
$u(0,x) \equiv 0$, and it belongs to the class
$G\{t,x\}_{(s,\sigma)}$ if and only if $(s,\sigma)$ satisfies 
$s \geq s_0^*$ and $\sigma \geq \sigma_0^*$.
\end{prop}
\begin{proof} As is seen in the first part of Section~\ref{secproofmaintheo}, $L(\lambda,\rho)$ satisfies (N) and (GP); therefore, the sufficiency follows from Theorem~\ref{maintheorem}. Our purpose is to show the necessity of the condition:
$s \geq s_0^*$ and $\sigma \geq \sigma_0^*$. 
\par
   Now, we suppose
that $u(t,x) \in \BC\{t,x\}_{(s,\sigma)}$ holds for some $s \geq 1$ 
and $\sigma \geq 1$. Let us show that $s \geq s_0^*$ and 
$\sigma \geq \sigma_0^*$ hold, in different steps.

In the discussion below, 
for two sequences of positive numbers $\{A_l \,;\, l \in \BN^* \}$
and $\{B_l \,;\, l \in \BN^* \}$ we write
$A_l \gtrsim B_l$ if there are $M>0$ and $H>0$ such that 
$A_l \geq MH^l B_l$ holds for all $l \in \BN^*$. In this case,
for $\rho>0$ we also write
$$
      \sum_{l \geq 1}A_l \rho^l 
                  \gtrsim \sum_{l \geq 1}B_l \rho^l.
$$
By Stirling's formula we have

\begin{lemma}\label{lema57}
The following statements hold:
\begin{itemize}
\item[(i)] We have $l^l \gtrsim l!$ and 
$l! \gtrsim l^l$.
\item[(ii)] For a fixed $n \in \BN^*$ we have
$(nl)! \gtrsim l!^n$ and $l! \gtrsim (nl)!^{1/n}$.
\item[(iii)] For fixed $m,n \in \BN^*$ we have $(nl+m)^l \gtrsim l!^n$.
\end{itemize}
\end{lemma}

   \noindent {\bf Step 1.} Let $u(t,x)=\sum_{k \geq 1} u_k(x) t^k \in (\BC[[x]])[[t]]$ be a formal solution of (\ref{e58}). Then, the coefficients $u_k(x)$ 
($k=1,2,\ldots$) are determined by the recurrence 
formulas
$$
     L(1,x \partial_x)u_1 
            = Ax^m + B x^p [x \partial_x]_{\beta}u_1 
$$
$$
     L(k,x \partial_x)u_k =
         B x^p k^h[x \partial_x]_{\beta}u_k
        + C(k-q)^j x^{\mu}\partial_x^{\alpha}u_{k-q},\qquad k \geq 2.
$$
The function $u_1(x)$ is given by $u_1(x)= \sum_{l \geq 0} A_{0,lp+m} x^{lp+m}
$ with 
\begin{equation}\label{e510}
    A_{0,lp+m}
    = \frac{A B^l [m]_{\beta}[p+m]_{\beta}
                     \cdots [(l-1)p+m]_{\beta}}
       {L(1,m) L(1,p+m) \cdots L(1,lp+m)}, \quad l \geq 0.
\end{equation}
Moreover, one can check that $u(t,x)$ has the form
\begin{equation}\label{e511}
    u(t,x)= \sum_{k \geq 0} u_{kq+1}(x)t^{kq+1} 
\end{equation}
and the coefficients $u_{kq+1}(x)$ ($k=1,2,\ldots$) are determined 
by the following recurrence formulas:
\begin{equation}\label{e512}
     L(kq+1,x \partial_x)u_{kq+1}=B x^p (kq+1)^h[x \partial_x]_{\beta}u_{kq+1}
     + C((k-1)q+1)^j x^{\mu}\partial_x^{\alpha}u_{(k-1)q+1}.
\end{equation}

\noindent{\bf Step 2.} We set $d_0=m$, and define $(l_{k-1},d_k)$ 
($k=1,2,\ldots$) by 
$$l_{k-1} = \min \{ l \in \BN:  lp+d_{k-1} \geq m+\alpha-\mu \},\qquad d_k = l_{k-1}p+d_{k-1}-\alpha+\mu,
$$
inductively on $k$.  

\begin{lemma}\label{lema58}
For any $k \in \BN^*$ we have the following results:
\begin{itemize}
\item[(i)] $0 \leq l_{k-1} \leq \alpha$ and 
$m \leq d_k \leq m+p$.
\item[(ii)] $lp+d_{k-1}-\alpha+\mu=(l-l_{k-1})p+d_k$.
\item[(iii)] If $C_{lp+d_{k-1}} \geq 0$ holds for all $l \in \BN$, we 
have
$$ x^{\mu} \partial_x^{\alpha}
                 \sum_{l \geq 0}C_{lp+d_{k-1}}x^{lp+d_{k-1}}\gg \sum_{l \geq 0}C_{lp+d_{k}+(\alpha-\mu)}x^{lp+d_k}\gg C_{d_{k}+(\alpha-\mu)}x^{d_k}= C_{l_{k-1}p+d_{k-1}} x^{d_k}.
$$
\end{itemize}
\end{lemma}
\begin{proof}
Since 
$\alpha p+d_0 \geq \alpha+m \geq \alpha+m-\mu$ holds, we have 
$l_0 \leq \alpha$. By the definition of $d_1$ we have 
$m \leq d_1 \leq m+p$. Let us show the general case of (i) by induction on $k$. Let $k \geq 2$ and suppose that $0 \leq l_{k-2} \leq \alpha$ and 
$m \leq d_{k-1} \leq m+p$ are known. Since 
$\alpha p+d_{k-1} \geq \alpha+m \geq \alpha+m-\mu$ holds, 
we have $l_{k-1} \leq \alpha$. By the definition of $d_k$
, and taking into account that $\alpha>\mu$ we get 
$$
    m \leq d_k \leq m+\max \{d_{k-1}-(m+\alpha-\mu), \, p \}\le m+\max \{p-(\alpha-\mu), \, p \}
            = m+p,
$$
which entails (i).

The result (ii) is clear from the definition of $d_k$.

In view of the second statement, and since 
$l_{k-1}p+d_{k-1}=d_k+\alpha-\mu \geq m+\alpha-\mu>\alpha$
holds, (iii) follows from the fact that
\begin{align*}
    &x^{\mu} \partial_x^{\alpha}
        \sum_{l \geq 0}C_{lp+d_{k-1}}x^{lp+d_{k-1}} = \sum_{{\scriptstyle l \geq 0}
              \atop{\scriptstyle (lp+d_{k-1} \geq \alpha)}}
       C_{lp+d_{k-1}} \frac{(lp+d_{k-1})!}{(lp+d_{k-1}-\alpha)!}
           x^{lp+d_{k-1}-\alpha+\mu} \\
    &\gg \sum_{l \geq l_{k-1}}
       C_{lp+d_{k-1}} 
           x^{lp+d_{k-1}-\alpha+\mu}\gg C_{d_k+(\alpha-\mu)} x^{d_k}
           = C_{l_{k-1}p+d_{k-1}} x^{d_k}.
\end{align*}

\end{proof}

\noindent{\bf Step 3.} We set $w_0(x)=u_1(x)$. By (iii) of Lemma~\ref{lema58} we have $       C x^{\mu}\partial_x^{\alpha}w_0
       \gg C A_{0,l_0p+m}x^{d_1}=K_1x^{d_1}$ with $K_1=C A_{0,l_0p+m}$.  Let us define $w_{1}(x)$ by the 
solution of
$$
     L(q+1,x \partial_x)w_{1}\\
     =B x^p (q+1)^h [x \partial_x]_{\beta}w_{1}
             + K_1x^{d_1}.
$$
Then we have $u_{q+1}(x) \gg w_{1}(x)= \sum_{l \geq 0} A_{1,lp+d_1}x^{lp+d_1}
$, where
$$
     A_{1,lp+d_1}
    = \frac{K_1 B^l (q+1)^{hl}[d_1]_{\beta}[p+d_1]_{\beta}
                 \cdots [(l-1)p+d_1]_{\beta}}
       {L(q+1,d_1) L(q+1,p+d_1) \cdots L(q+1,lp+d_1)}, 
        \quad l \geq 0.
$$
Regarding (iii) of Lemma~\ref{lema58} we derive $C(q+1)^jx^{\mu}\partial_x^{\alpha}w_{1}
       \gg C (q+1)^jA_{1,l_1p+d_1}x^{d_2}
       = K_2x^{d_2}
$ with $K_2=C (q+1)^jA_{1,l_1p+d_1}$.  

The construction follows recursively. Assume $w_{k-1}(x)= \sum_{l \geq 0} A_{k-1,lp+d_{k-1}}x^{lp+d_{k-1}}$. By setting $K_k=C ((k-1)q+1)^jA_{k-1,l_{k-1}p+d_{k-1}}$ and defining $w_{k}(x)$ by the solution of 
$$
     L(kq+1,x \partial_x)w_{k}\\
     =B x^p (kq+1)^h [x \partial_x]_{\beta}w_{k}
             + K_kx^{d_k},
$$
then we have $u_{kq+1}(x) \gg w_{k}(x)$ and $w_{k}(x)= \sum_{l \geq 0} A_{k,lp+d_k}x^{lp+d_k}
$, where
\begin{equation}\label{e513}
    A_{k,lp+d_k} = \frac{K_k B^l (kq+1)^{hl}[d_k]_{\beta}[p+d_k]_{\beta}
                 \cdots [(l-1)p+d_k]_{\beta}}
       {L(kq+1,d_k) L(kq+1,p+d_k) \cdots L(kq+1,lp+d_k)}, 
       \quad l \geq 0.
\end{equation}

\noindent {\bf Step 4.} By the discussion in Step 3 we have the following
result. We can define $(K_k,A_{k,lp+d_k})$ 
($k \in \BN$ and $l \in \BN$) inductively on $k$: 
$K_0=A$, $A_{0,lp+d_0}$ as in (\ref{e510}), and for $k \geq 1$ we set
$K_k=C ((k-1)q+1)^jA_{k-1,l_{k-1}p+d_{k-1}}$ and $A_{k,lp+d_k}$ as in (\ref{e513}). We define
$$
    w(t,x)= \sum_{k \geq 0, l \geq 0}
                 A_{k,lp+d_k}t^{kq+1}x^{lp+d_k}.
$$
Then we have $u(t,x) \gg w(t,x)$.  

\begin{lemma}\label{lema59}
In the previous situation, the following statements hold:
\begin{itemize}
\item[(i)] There are $C_1>0$ and $H_1>0$ such that
\begin{equation}\label{e514}
     K_k \geq C_1{H_1}^k \frac{1}{k!^{m(\alpha+1)}}
     \quad \mbox{for any $k=0,1,2,\ldots$}.  
\end{equation}
\item[(ii)] Let $s_i$ be as in Figure~\ref{e59}. For $l \in \BN^*$ we set 
$k_l = [((lp+m+p)^{s_i}-1)/q]$.  Then, there are $C_2>0$, 
$H_2>0$, $a>0$ and $b>0$ such that
\begin{equation}\label{e515}
    A_{k_l,lp+d_{k_l}} \geq C_2{H_2}^l 
    \frac{l!^d}{[al^{s_i}+b]!^{m(\alpha+1)}}
    \quad \mbox{for any $l \in \BN^*$}.   
\end{equation}
\end{itemize}
\end{lemma}
\begin{proof} 
The definition of $K_k, A_{k,l_kp+d_k}$ entails
$$   K_{k+1}=C (kq+1)^jA_{k,l_kp+d_k} \geq   \frac{K_k C B^{l_k}}
          {L(kq+1,l_kp+d_k)^{l_k+1}}.
$$
Since $0 \leq l_k \leq \alpha$, we have
$B^{l_k} \geq b_1$ ($k \in \BN$) for some $b_1>0$. Since 
$l_kp+d_k \leq \alpha p+m+p$, by (ii) of Lemma~\ref{lema52} we have
$L(kq+1,l_kp+d_k) \leq b_2 (k+1)^m$ ($k \in \BN$) for some 
$b_2 \geq 1$. Therefore, 
$$
   K_{k+1} \geq  \frac{K_k C b_1}
          {(b_2(k+1)^m)^{l_k+1}}
        \geq  \frac{K_kC b_1}
          {(b_2(k+1)^m)^{\alpha+1}}
$$
for any $k \in \BN$. Since $K_0=A$, the previous inequality leads us to (\ref{e514}).

Let us show the second statement. By (\ref{e513}), (\ref{e514}) and (i) of Lemma~\ref{lema52} we have
\begin{equation}\label{e517}
    A_{k,lp+d_k}\geq \frac{C_1{H_1}^k}{k!^{m(\alpha+1)}} \times 
          \frac{B^l (kq+1)^{hl}\, l!^{\beta}}
                           {L(kq+1, lp+d_k)^{l+1}} 
    \geq \frac{C_1{H_1}^k}{k!^{m(\alpha+1)}} \times 
          \frac{B^l (kq+1)^{hl}\, l!^{\beta}}
                    {(c_1\phi(kq+1, lp+m+p))^{l+1}}.
\end{equation}
\par
   Since $k_l=[((lp+m+p)^{s_i}-1)/q]$, we have 
$k_lq+1 \leq (lp+m+p)^{s_i} \leq lp+m+p$ and so by (iii) of 
Lemma~\ref{lema52} we have
\begin{align*}
     \phi(k_lq+1, lp+m+p)
     &\leq \phi([(lp+m+p)^{s_i}], lp+m+p) \\
     &\leq c_3 (lp+m+p)^{s_im_i}(lp+m+p)^{n_i}, 
        \quad l \in \BN^*
\end{align*}
for some $c_3>0$.

 Similarly, we have $ {H_1}^{k_l} \geq (\min \{1,H_1\})^{(lp+m+p-1)/q}$. Since $0 \leq s_i<1$, we have
$(lp+m+p)^{s_i} \leq (lp)^{s_i}+(m+p)^{s_i}$ and so by setting
$a=p^{s_i}/q$ and $b=((m+p)^{s_i}-1)/q$ we have 
$k_l \leq [al^{s_i}+b]$. 
\par
   Since $(lp+m+p)^{s_i} \geq 1$ we have
$k_l \geq 0$ and so $k_lq+1 \geq 1$. Taking into account the previous statements, and  $k_lq+1 \geq (lp+m+p)^{s_i}-q$, we derive
$$    k_lq+1 \geq \max\{1, (lp+m+p)^{s_i}-q \} \geq \max\{1, (lp)^{s_i}-q \}
          \geq \frac{p^{s_i}}{q+1} l^{s_i}.
$$
In the last inequality we have used the fact that
$\max \{1,x-q \} \geq x/(q+1)$ for any $x \in \BR$.

Thus, by applying these estimates to (\ref{e517}), under the condition 
$k_l=[((lp+m+p)^{s_i}-1)/q]$ we have 
\begin{align*}
    A_{k_l,lp+d_{k_l}}
    &\geq \frac{C_1 (\min\{1,H_1 \})^{(lp+m+p-1)/q}}
                    {[al^{s_i}+b]!^{m(\alpha+1)}} \frac{B^l (p^{s_i}/(q+1))^{hl}(l^{s_i})^{hl}\, l!^{\beta}}
               {(c_0c_3 (lp+m+p)^{s_im_i}(lp+m+p)^{n_i})^{l+1}} \\
    &\gtrsim \frac{l!^{s_ih} l!^{\beta}}{[al^{s_i}+b]!^{m(\alpha+1)}
                \times l!^{s_im_i+n_i}}
       = \frac{l!^d}{[al^{s_i}+b]!^{m(\alpha+1)}}.
\end{align*}
In the above, we have used that $d=\beta-n_i-s_i(m_i-h)$. This proves (\ref{e515}).  
\end{proof}

{\bf Step 5.} Let us show the condition: $\sigma \geq \sigma_0^*$.
Since $u(t,x) \in \BC\{t,x\}_{(s,\sigma)}$ is supposed and since
$u(t,x) \gg w(t,x)$ is known, we have 
$w(t,x) \in \BC\{t,x\}_{(s,\sigma)}$, that is, 
$$
   \sum_{k \geq 0, l \geq 0} \frac{A_{k,lp+d_k}}
         {(kq+1)!^{s-1}(lp+d_k)!^{\sigma-1}} \rho^{kq+1}
       \rho^{lp+d_k} <\infty
$$
holds for some $0<\rho \leq 1$. 
\par
   If we set
$k_l=[((lp+m+p)^{s_i}-1)/q]$ we have $k_lq+1 \leq (lp+m+p)^{s_i}
\leq (lp)^{s_i}+(m+p)^{s_i}=a_1l^{s_i}+b_1$ with $a_1=p^{s_i}$
and $b_1=(m+p)^{s_i}$. Therefore, by (ii) of Lemma~\ref{lema59} we have
\begin{align}
   \infty &> \sum_{l \geq 1, k_l=[((lp+m+p)^{s_i}-1)/q]}
       \frac{A_{k_l,lp+d_{k_l}} 
                       {\rho}^{k_lq+1} {\rho}^{lp+d_{k_l}}}
         {(k_lq+1)!^{s-1}(lp+d_{k_l})!^{\sigma-1}}  \label{e518}\\
    &\geq \sum_{l \geq 1} 
      \frac{C_2 {H_2}^l l!^d {\rho}^{lp+m+p} {\rho}^{lp+m+p}}
        {[al^{s_i}+b]!^{m(\alpha+1)}
         [a_1l^{s_i}+b_1]!^{s-1}(lp+m+p)!^{\sigma-1}} \nonumber\\
    &\gtrsim \sum_{l \geq 1} 
      \frac{l!^d {\rho}^{lp+m+p} {\rho}^{lp+m+p}}
        {[al^{s_i}+b]!^{m(\alpha+1)}
         [a_1l^{s_i}+b_1]!^{s-1} l!^{p(\sigma-1)}}.\nonumber
\end{align}
If $d> p(\sigma-1)$ holds, we can derive a contradiction in the 
following way: if we set $2\epsilon=d-p(\sigma-1)>0$, by (\ref{e518}) 
and (i) of Lemma~\ref{lema55} we have
$$   \infty > \sum_{l \geq 1} \frac{l!^{\epsilon}}
            {[al^{s_i}+b]!^{m(\alpha+1)}
         [a_1l^{s_i}+b_1]!^{s-1}} l!^{\epsilon}
         {\rho_1}^{l} 
    \geq C_1 \sum_{l \geq 1} l!^{\epsilon}
         {\rho_1}^{l} = \infty
$$
for some $0<\rho_1<\rho$ and some $C_1>0$. Then, 
$\sigma \geq 1+d/p=\sigma_0^*$.

   {\bf Step 6.} We express every coefficient of $u(t,x)$ in (\ref{e511})  in the form
$$
     u_{kq+1}(x)= \sum_{l \geq 0} u_{kq+1,l}x^l.
$$
By (\ref{e512}) we have $L(kq+1,x \partial_x)u_{kq+1}
    \gg C((k-1)q+1)^j x^{\mu} \partial_x^{\alpha} u_{(k-1)q+1}
$, which entails
$$
    u_{kq+1,l} 
    \geq \frac{C((k-1)q+1)^j(l-\mu+1)^{\alpha}}{L(kq+1,l)}
          u_{(k-1)q+1,(\alpha-\mu)+l}, \quad l \geq \mu.
$$
Hence, by using this estimate $lp$-times and by the estimate
$u_{kq+1}(x) \gg w_k(x)$ we have
\begin{align}
    u_{(k+lp)q+1,d_k} &\geq  
       \frac{C^{lp}\prod_{n=0}^{lp-1}((k+n)q+1)^j
          \prod_{n=0}^{lp-1}(n(\alpha-\mu)+d_k-\mu+1)^{\alpha}}
      {\prod_{n=0}^{lp-1}
           L((k+lp-n)q+1, n(\alpha-\mu)+d_k)} u_{kq+1, lp(\alpha-\mu)+d_k}\label{e519} \\
    &\geq
       \frac{C^{lp} (lp)!^j(lp)!^{\alpha}}
           {L((k+lp)q+1, lp(\alpha-\mu)+d_k)^{lp}}
                    A_{k, lp(\alpha-\mu)+d_k}\nonumber \\
    &\geq
       \frac{C^{lp} (lp)!^j(lp)!^{\alpha}}
           {(c_1 \phi((k+lp)q+1, lp(\alpha-\mu)+d_k))^{lp}}
                    A_{k, lp(\alpha-\mu)+d_k}.\nonumber
\end{align}
Here, we set $k_l=[((lp(\alpha-\mu)+m+p)^{s_i}-1)/q]$ 
($l \in \BN^*$).  Since $m_r+n_r\leq m$ ($r=1,\ldots,p$) hold,
we have
\begin{equation}\label{e520}\phi((k_l+lp)q+1, lp(\alpha-\mu)+d_{k_l}) \leq \sum_{r=1}^p
       \bigl((lp(\alpha-\mu)+m+p)^{s_i}+lpq \bigr)^{m_r}
               \bigl(lp(\alpha-\mu)+m+p \bigr)^{n_r}\leq c_3(lp)^m,
\end{equation}							
for $l \in \BN^*$, for some $c_3>0$. Therefore, under the condition
$k_l=[((lp(\alpha-\mu)+m+p)^{s_i}-1)/q]$, by applying (\ref{e520}) and 
(\ref{e515}) to (\ref{e519}) we have
\begin{align*}
    u_{(k_l+lp)q+1,d_{k_l}} 
    &\geq
       \frac{C^{lp} (lp)!^j(lp)!^{\alpha}}{(c_1c_3(lp)^m)^{lp}}
          \times C_2 {H_2}^{l(\alpha-\mu)}
        \frac{(l(\alpha-\mu))!^d}
             {[a(l(\alpha-\mu))^{s_i}+b]!^{m(\alpha+1)}}\\
    &\gtrsim
       \frac{(lp)!^j(lp)!^{\alpha}}{(lp)!^m}
          \frac{(lp)!^{(d/p)(\alpha-\mu)}}
             {[a_0(lp)^{s_i}+b]!^{m(\alpha+1)}}
\end{align*}
where $a_0=a((\alpha-\mu)/p)^{s_i}$.  Thus, by setting
$K=j+\alpha+(d/p)(\alpha-\mu)-m$ and 
$k_l=[((lp(\alpha-\mu)+m+p)^{s_i}-1)/q]$ we have
\begin{equation}\label{e521}
   u_{(k_l+lp)q+1,d_{k_l}}
    \gtrsim \frac{(lp)!^K}{[a_0(lp)^{s_i}+b]!^{m(\alpha+1)}},
    \quad l \in \BN^*. 
\end{equation}

   {\bf Step 7.} Let us write
$k_l=[((lp(\alpha-\mu)+m+p)^{s_i}-1)/q]$ (for $l \in \BN^*$). Then, it is straight that $l_1 \ne l_2$ implies 
$k_{l_1}+l_1p \ne k_{l_2}+l_2p$. 

\vspace{3mm}

{\bf Step 8.} Lastly, by using (\ref{e521}) let us show the 
condition $s \geq s_0^*$.
Since $u(t,x) \in \BC\{t,x \}_{(s,\sigma)}$ is 
supposed, there is a $0<\rho \leq 1$ such that
$$
    \sum_{k \geq 0, l \geq 0} 
        \frac{u_{kq+1,l}}{(kq+1)!^{s-1}l!^{\sigma-1}}
           \rho^{kq+1}\rho^l < \infty.
$$
Therefore, by (\ref{e521}) and Step 7 we have
\begin{align*}
   \infty &> \sum_{l \geq 1, k=k_l} 
             \frac{u_{(k+lp)q+1,d_k}}
                {((k+lp)q+1)!^{s-1} d_k!^{\sigma-1}}
             {\rho}^{(k+lp)q+1}\rho^{d_k} \\
    &\gtrsim \sum_{l \geq 1}
          \frac{(lp)!^K}{[a_0(lp)^{s_i}+b]!^{m(\alpha+1)}
                    ((k_l+lp)q+1)!^{s-1} d_k!^{\sigma-1}}
               {\rho}^{(k_l+lp)q+1}\rho^{d_k}.
\end{align*}
We have
\begin{align*}
   ((k_l+lp)q+1) &=(k_lq+1) + lpq \leq (lp(\alpha-\mu)+m+p)^{s_i}+lpq \\
    &\leq (lp(\alpha-\mu))^{s_i}+ (m+p)^{s_i}+lpq \leq a_2 (lpq)^{s_i} +lpq \leq (a_2+1)(lpq),
\end{align*}
with $a_2=((\alpha-\mu)/q)^{s_i}+(m+p)^{s_i}$. Hence, by 
taking a smaller $0<\rho_1<\rho$ we have
\begin{equation}\label{e522}
   \infty > \sum_{l \geq 1}
          \frac{(lpq)!^{K/q}}{[a_0(lp)^{s_i}+b]!^{m(\alpha+1)}
            [a_2 (lpq)^{s_i} +lpq]!^{s-1}}
               {\rho_1}^{(a_2+1)lpq}.  
\end{equation}
If $K/q>(s-1)$, we derive a contradiction. More precisely, set $\epsilon=(K/q-(s-1))/3$, by (\ref{e522}) and Lemma~\ref{lema55} we have
\begin{align*}
   \infty &> \sum_{l \geq 1}
      \frac{(lpq)!^{\epsilon}}{[a_0(lp)^{s_i}+b]!^{m(\alpha+1)}}
      \frac{(lpq)!^{\epsilon+(s-1)}}{
            [a_2 (lpq)^{s_i} +lpq]!^{s-1}} \times (lpq)!^{\epsilon}
               {\rho_1}^{(a_2+1)lpq} \\
    &\geq C_2 \sum_{l \geq 1}(lpq)!^{\epsilon}
            {\rho_1}^{(a_2+1)lpq}  = \infty
\end{align*}
for some $C_2>0$. This entails that $s \geq s_0^*$, and completes the proof of Proposition~\ref{prop56}.

\end{proof}	

In the following, we state a variant of Proposition~\ref{prop56}. Let us consider the equation
\begin{equation}\label{e523}
    L(t\partial_t,x \partial_x)u
     = Axt + Bx^p(t\partial_t)^h(x\partial_x)^{\beta}u
          +C t^qx^{\mu}(t\partial_t)^j \partial_x^{\alpha}u
\end{equation}
under the same assumptions h${}_1$)$\sim$h${}_4$) as 
in (\ref{e58}). 
Let $\sigma_0^*$ and $s_0^*$ be as in Proposition~\ref{prop56}. Then, an analogous argument as above, the next result is attained.

\begin{prop}\label{prop510}
The equation (\ref{e523}) has a unique formal solution $u(t,x) \in \BC[[t,x]]$ satisfying $u(0,x) \equiv 0$, and it belongs to the class
$G\{t,x\}_{(s,\sigma)}$ if and only if $(s,\sigma)$ satisfies 
$s \geq s_0^*$ and $\sigma \geq \sigma_0^*$.
\end{prop}

\section{A generalization}\label{sec6}

Let $C(x;\lambda,\rho)$ be as in (\ref{e13}), ${\mathcal M}$ be a 
finite subset of $\BN \times \BN$, and let 
$\bz=\{z_{j,\alpha} \}_{(j,\alpha) \in {\mathcal M}}$ be the complex
variables in $\BC^N$ (with $N=\# {\mathcal M}$). We consider
\begin{equation}\label{e61}
    C(x; t\partial_t,x \partial_x)u = a(x)t + G_2 \bigl(t,x,
         \bigl\{(t\partial_t)^j \partial_x^{\alpha}u 
         \bigr\}_{(j,\alpha) \in {\mathcal M}} \bigr),
\end{equation}

\noindent where $G_2(t,x,\bz)$ is a holomorphic function in a neighborhood
of $(0,0,0) \in \BC_t \times \BC_x \times \BC_{\bz}^N$ whose Taylor 
expansion in $(t,\bz)$ has the form
$$
     G_2(t,x,\bz)= \sum_{i+|\bnu| \geq 2}
         g_{i,\bnu}(x) t^i \bz^{\bnu}
$$
with $\bnu=\{\nu_{j,\alpha} \}_{(j,\alpha) \in {\mathcal M}} \in \BN^N$,
$|\bnu|=\sum_{(j,\alpha) \in {\mathcal M}}\nu_{j,\alpha}$ and 
$\bz^{\bnu}=\prod_{(j,\alpha) \in {\mathcal M}}
         {z_{j,\alpha}}^{\nu_{j,\alpha}}$.
\par
   If ${\mathcal M}=I_m$, equation (\ref{e61}) coincides with (\ref{e11})
(or (\ref{e15})).
We can define the 
irregularity $\sigma_0$ of (\ref{e61}) at $x=0$ in the same way as (\ref{e24}). For $\mu \in \BN$ we set 
$$
    J_{\mu}=\{(i,\bnu) \in \BN \times \BN^N \,;\,
         i+|\bnu| \geq 2, |\bnu| \geq 1, 
                 (\partial_x^{\mu}g_{i,\bnu})(0) \ne 0 \}.
$$
For $\mu \in \BN$ and 
$\bnu=\{\nu_{j,\alpha} \}_{(j,\alpha) \in {\mathcal M}}$ satisfying 
$|\bnu| \geq 1$ we set

$$
K_{\bnu}=\{(j,\alpha) \in {\mathcal M} \,;\, \nu_{j,\alpha}>0 \},\quad m_{\bnu,\mu} = \max_{(j,\alpha) \in K_{\bnu}}
         \bigl( j+ \max\{\alpha, \mu+\sigma_0(\alpha-\mu)\}
         \bigr),$$
				$$s_0= 1+ \max \biggl[\,0 \,, \sup_{\mu \geq 0} \,
            \Bigl( \sup_{(i,\bnu) \in J_{\mu}}
              \frac{m_{\bnu,\mu}-m}{i+|\bnu|-1} \Bigr)
             \biggr].
$$

The same arguments as in Section~\ref{sec4} apply to obtain the following results.
	
\begin{theo}\label{teo61}
Suppose the conditions (N) and (GP) hold. Then, the equation (\ref{e61}) has a unique formal solution $u(t,x) \in \BC[[t,x]]$ satisfying $u(0,x) \equiv 0$, and it 
belongs to the class $G\{t,x \}_{(s,\sigma)}$ for any $s \geq s_0$ 
and $\sigma \geq \sigma_0$.
\end{theo}

\noindent{\bf Remark.}
\begin{itemize}
\item[(i)] The index $s_0$ is also expressed 
in the form
$$
     s_0= 1 + \max \biggl[ \, 0, \, \sup_{\mu \geq 0}
          \Bigl( \max_{(j,\alpha) \in {\mathcal M}}
              \frac{j+ \max\{\alpha, \mu+ \sigma_0(\alpha-\mu) \}-m}
                   {L_{\mu,j,\alpha}} \Bigr) 
           \biggr]
$$
where $L_{\mu,j,\alpha} = \mathrm{val} \bigl(
          ( \partial_{z_{j,\alpha}}
           \partial_x^{\, \, \mu} G_2)(t,0,\bz) \bigr)$
($\mu \in \BN$ and $(j,\alpha) \in {\mathcal M}$).
\item[(ii)] If $\sigma_0=1$, we have
$$
     s_0= 1 + \max \biggl[ \, 0, \, \sup_{\mu \geq 0}
          \Bigl( \max_{(j,\alpha) \in {\mathcal M}}
              \frac{j+\alpha -m}{L_{\mu,j,\alpha}} \Bigr) 
           \biggr].
$$
\end{itemize}

Hence, if $\sigma_0=1$ and 
${\mathcal M} \subset \{(j,\alpha) \,;\, j+\alpha \leq m \}$, 
we have $s_0=1$ and the formal power series solution $u(t,x)$ is 
convergent in a neighborhood of $(0,0) \in \BC_t \times \BC_x$. 

As to the optimality, to get the same result as in Theorem~\ref{teo27} we need some additional condition. We set:
\begin{align*}
    &M=\max\{\alpha \,;\, (j,\alpha) \in {\mathcal M} \}, \qquad {\mathcal M}_{\mu}=\{(j,\alpha) \in {\mathcal M} \,;\,
      \alpha \geq \mu \}, \quad 0 \leq \mu \leq M, \\
    &s_1= 1 + \max \biggl[ \, 0, \, \max_{0 \leq \mu \leq M}
          \Bigl( \max_{(j,\alpha) \in {\mathcal M_{\mu}}}
              \frac{j+\mu+ \sigma_0(\alpha-\mu)-m}
                   {L_{\mu,j,\alpha}} \Bigr) 
           \biggr].
\end{align*}
We note that if $\alpha \geq \mu$ we have 
$\max\{\alpha, \mu+ \sigma_0(\alpha-\mu) \}=\mu+ \sigma_0(\alpha-\mu)$.
\par
   In general, we have $s_0 \geq s_1$. By the same argument 
as in Section~\ref{secproofmaintheo} we have

\begin{theo}\label{teo63} Suppose the condition $s_0=s_1$.
In addition, assume the following conditions:
\begin{itemize}
\item[c-1)] $(\partial_x^{\, m}a)(0)>0$, $(\partial_x^{\,\,\mu}a)(0)>0$ for 
             $0 \leq \mu \leq M$ and $a(x) \gg 0$;
\item[c-2)]  $c_{j,\alpha}(0) \leq 0$  for any
              $(j,\alpha) \in I_m$, 
\item[c-3)]  $c_{j,\alpha}(x)-c_{j,\alpha}(0) \gg 0$ for any
              $(j,\alpha) \in I_m$, 
\item[c-4)] $g_{i,\bnu}(x) \gg 0$ for any $(i,\bnu)$ with 
             $i+|\bnu| \geq 2$.
\end{itemize}

Then, equation (\ref{e61}) has a unique formal solution 
$u(t,x) \in \BC[[t,x]]$ satisfying $u(0,x) \equiv 0$, and it belongs 
to the class $G\{t,x \}_{(s,\sigma)}$ if and only if $(s,\sigma)$
satisfies $s \geq s_0$ and $\sigma \geq \sigma_0$.
\end{theo}

In the case $s_0 > s_1$, our index $s_0$ is not optimal in general, as it is seen in the following example.

\vspace{3mm}

\noindent{\bf Example:} Let us consider
\begin{equation}\label{e62}
    t\partial_t u= xt + x(x \partial_x)u + tx^3 \partial_x^2u.
\end{equation}
In this case, we have ${\mathcal M}=\{(0,2) \}$, $\sigma_0=2$, $s_0=2$
and $s_1=1$. Equation (\ref{e62}) has a unique 
formal solution $u(t,x) \in \BC[[t,x]]$ satisfying $u(0,x) \equiv 0$, 
and it belongs to the class $G\{t,x \}_{(1,2)}$.
\begin{proof} We set $u(t,x)=\sum_{k \geq 1}u_k(x)t^k$. Then,  $u_k(x) \in \BC[[x]]$ ($k=1,2,\ldots$) are uniquely 
determined inductively on $k$ by the relations:
$u_1=x+x(x\partial_x)u_1$ and for $k \geq 2$
\begin{equation}\label{e63}
    ku_k = x(x \partial_x)u_k + x^3 \partial_x^2u_{k-1}. 
\end{equation}
In addition, for any $k=1,2,\ldots$ we have $u_k(x) \gg 0$ and 
\begin{equation}\refstepcounter{equation}\label{e64}
     {\mathcal B}_2[u_k](x) \ll \frac{2^{k-1}}{(1-x)^{2k-1}}.\tag*{(\theequation)$_k$}
\end{equation}
This proves that $u(t,x) \in G\{t,x \}_{(1,2)}$. 

The proof of \ref{e64} is as follows. The case $k=1$ is verified
by a direct calculation. Let $k \geq 2$ and suppose that this property holds for $k-1$. Then, we have
\begin{equation}\label{e65}
    {\mathcal B}_2[x^3 \partial_x^2u_{k-1}]
    \ll x^2 \partial_x{\mathcal B}_2[u_{k-1}]
      \ll \frac{2^{k-2} x^2(2k-3)}{(1-x)^{2k-2}}
      \ll \frac{2^{k-2}(2k-3)}{(1-x)^{2k-2}}. 
\end{equation}
Since ${\mathcal B}_2[x(x\partial_x)u_{k}]
            \ll x {\mathcal B}_2[u_{k}]$, by (\ref{e63}) and (\ref{e65}) we have
$$ {\mathcal B}_2[u_{k}]
   \ll \frac{1}{k-x} {\mathcal B}_2[x^3 \partial_x^2u_{k-1}](x)
   \ll \frac{1}{k(1-x)}\frac{2^{k-2}(2k-3)}{(1-x)^{2k-2}}
   \ll  \frac{2^{k-1}}{(1-x)^{2k-1}}.
$$
This proves \ref{e64}.
\end{proof}

%
%
\end{document}